\documentclass[12pt, reqno]{amsart}
\usepackage{amsmath, amsthm, amscd, amsfonts, amssymb, graphicx, color, mathrsfs}
\usepackage[bookmarksnumbered, colorlinks, plainpages]{hyperref}
\usepackage[all]{xy}
\usepackage{slashed}

\usepackage{soul}
\usepackage{cancel}
\usepackage{ulem}

\newtheorem{theorem}{Theorem}[section]
\newtheorem{lemma}[theorem]{Lemma}

\newtheorem{corollary}[theorem]{Corollary}
\theoremstyle{definition}

\theoremstyle{remark}
\newtheorem{remark}[theorem]{Remark}
\numberwithin{equation}{section}

\begin{document}
\setcounter{page}{1}

\title[ Psdos in subelliptic Besov spaces on compact Lie groups ]{ Boundedness of pseudo-differential operators in subelliptic Sobolev and Besov spaces on compact Lie groups}

\author[D. Cardona]{Duv\'an Cardona}
\address{
  Duv\'an Cardona:
  \endgraf
  Department of Mathematics
  \endgraf
  Pontificia Universidad Javeriana.
  \endgraf
  Bogot\'a
  \endgraf
  Colombia
  \endgraf
  {\it E-mail address} {\rm duvanc306@gmail.com}
  }

\author[M. Ruzhansky]{Michael Ruzhansky}
\address{
  Michael Ruzhansky:
  \endgraf
  Department of Mathematics: Analysis, Logic and Discrete Mathematics
  \endgraf
  Ghent University, Belgium
  \endgraf
 and
  \endgraf
  School of Mathematics
  \endgraf
  Queen Mary University of London
  \endgraf
  United Kingdom
  \endgraf
  {\it E-mail address} {\rm ruzhansky@gmail.com}
  }

\subjclass[2010]{Primary {35S30; Secondary 58J40}.}

\keywords{Sub-Laplacian, Compact Lie group,  Besov spaces, Nikolskii's Inequality}

\thanks{The second author was supported in parts by the FWO Odysseus Project, EPSRC grant
EP/R003025/1 and by the Leverhulme Grant RPG-2017-151}

\begin{abstract}
In this paper we  investigate the Besov spaces on compact Lie groups in a subelliptic setting, that is, associated with a family of vector fields, satisfying the H\"ormander condition, and their corresponding sub-Laplacian. Embedding properties between subelliptic Besov spaces and Besov spaces associated to the Laplacian on the group are proved. We link the description of subelliptic Sobolev spaces with the matrix-valued quantisation procedure of  pseudo-differential operators in order  to provide subelliptic Sobolev and Besov estimates for operators in the H\"ormander classes. Interpolation properties between Besov spaces and Triebel-Lizorkin spaces are also investigated.
\end{abstract} \maketitle

\tableofcontents
\section{Introduction}
In this work,  Besov spaces  associated to sub-Laplacians in the context of compact Lie groups are studied. We refer to them as subelliptic Besov spaces and we obtain some of their embedding properties  in terms of the Hausdorff dimensions of compact Lie groups  defined by the control distances associated to the sub-Laplacians. The novelty of our approach is that we use a description for these spaces in terms of the matrix-valued symbol associated to pseudo-differential operators, which we exploit  particularly for sub-Laplacians. Our methods are based on the mapping properties of  Fourier multipliers and global pseudo-differential operators consistently developed in the works \cite{RuzhanskyDelgado2017,RuzhanskyWirth2015}, as well as in the matrix-valued functional calculus established in \cite{RuzhanskyWirth2014}.

The Besov spaces $B^s_{p,q}(\mathbb{R}^n)$   arose from attempts to unify the various definitions of
several fractional-order Sobolev spaces. Taibleson studied the generalised H\"older-Lipschitz spaces $\Lambda^{s}_{p,q}(\mathbb{R}^n)$ and these spaces were called Besov spaces in honor of O. V.  Besov who obtained a trace theorem and important embedding properties for them (see  Besov \cite{Besov1,Besov2}).  Dyadic decompositions (defined by the spectrum of the Laplacian) for Besov spaces on $\mathbb{R}^n$ were introduced by J. Peetre as well as other embedding properties where obtained (see Peetre \cite{Peetre1,Peetre2}). We refer the reader to the works of Triebel \cite{Triebel1983} and \cite{Triebel2006} for a complete background on the subject as well as a complete description of the historical facts about Besov spaces and other function spaces.
 
Throughout this work   $G$ is a compact Lie group and the positive sub-Laplacian $\mathcal{L}=-(X_1^2+\cdots +X_k^2),$ will be considered in such a way that the system of vector fields $X=\{X_i\}$ satisfies the H\"ormander condition. With a dyadic partition $\{\psi_{\ell}\}_{\ell\in\mathbb{N}},$ for the spectrum  of the operator $(1+\mathcal{L})^{\frac{1}{2}},$ Besov spaces $B^{s,\mathcal{L}}_{p,q}(G)$ associated to $\mathcal{L}$ can be defined by the norm
\begin{equation}\label{definitionBesovIntro}
\Vert f\Vert_{B^{s,\mathcal{L}}_{p,q}(G)}= \left\Vert   \{2^{\ell  s}\Vert  \psi_\ell((1+\mathcal{L})^\frac{1}{2})f\Vert_{L^p(G)} \}_{\ell\in\mathbb{N}}\right\Vert_{\ell^q(\mathbb{N})}.
\end{equation} The parameter $p$
in \eqref{definitionBesovIntro} measures the type of integrability of the function $f,$ while the parameter $s$ is the maximal order of regularity of $f$. In terms of the sub-Laplacian, this regularity order can be measured by the action of the operator  $(1+\mathcal{L})^\frac{1}{2} $ on the information that the Fourier expansion of $f$ provides. This means that we can approximate the function $(1+\mathcal{L})^\frac{1}{2}\psi_\ell((1+\mathcal{L})^\frac{1}{2})f $ by  another more simple one, in this case given by $2^{\ell s} \psi_\ell((1+\mathcal{L})^\frac{1}{2})f.$ This `approximation' when $\ell\rightarrow\infty$ is understood in the $L^p$-norm and the $q$-parameter is a measure of the speed in which this approximation occurs. If it occurs with speed $O(\ell^{-\varepsilon})$ where $\varepsilon>1/q$ then we classify  $f$ to be in the space $B^{s,\mathcal{L}}_{p,q}(G).$ If this approximation occurs with $O(1)$ speed, we can think that $f\in B^{s,\mathcal{L}}_{p,\infty}(G). $
This general analysis can be applied to every (compact or non-compact) Lie group $G$  where Besov spaces are defined. The main feature of these spaces is that it allow us to better understand other functions spaces. It is the case of  the Hilbert space $L^2(G)$ or another Sobolev spaces with integer  or fractional order (see e.g.  Theorem \ref{Emb1}).

If $\mathcal{L}_G=-(X_{1}^2+\cdots +X_{n}^2),$ $n:=\dim(G),$ is the Laplacian on the group, the Fourier description for Besov spaces $B^s_{p,q}(G)\equiv B^{s,\mathcal{L}_G}_{p,q}(G)$ was consistently  developed by Nursultanov, Tikhonov and the second author in the works \cite{NurRuzTikhBesov2015} and  \cite{NurRuzTikhBesov2017} for  compact Lie groups and general compact homogeneous manifolds.
For obtaining non-trivial embedding properties for these spaces, the authors developed the Nikolskii inequality, which was established in terms of the dimension of the group $n$ and the Weyl-eigenvalue counting formulae for the Laplacian. To develop similar properties in the subelliptic context, we obtain a suitable subelliptic version of the Nikolskii inequality but in our case this will be presented in terms of the Hausdorff dimension $Q$ defined by the control distance associated to   the sub-Laplacian under consideration (see Eq. \eqref{Hausdorff-dimension}).

In the subelliptic framework, there are important  differences between  subelliptic Besov spaces and the Besov spaces associated to the Laplacian. Indeed, if $\{X_i\}$ is a system of vector fields satisfying the H\"ormander condition of order $\kappa,$ (this means that their iterated commutators of length $\leq \kappa$ span the Lie algebra $\mathfrak{g}$ of $G$) the sub-Laplacian  $\mathcal{L}=-(X_1^2+\cdots +X_k^2),$ $1\leq k<n,$ acting on smooth functions, derives in the directions of the vector fields $X_{i},$ $1\leq i\leq n,$ but the unconsidered directions $X_{i},$ $k+1 \leq i\leq n$ provide a loss of regularity reflected in the following embeddings
\begin{equation}
B^{s}_{p,q}(G)\hookrightarrow B^{s,\mathcal{L}}_{p,q}(G) \hookrightarrow B^{s_{\kappa,\varkappa}}_{p,q}(G),\,\,s_{\kappa,\varkappa}:=\frac{s}{\kappa}-\varkappa\left(1-\frac{1}{\kappa}\right)\left|\frac{1}{2}-\frac{1}{p}\right|,
\end{equation} that we will prove in Theorem \ref{SubBesovvsEllipBesov} for all $s,$ with $0<s<\infty$. Here we have denoted by $\varkappa:=[n/2]+1,$ the smallest  integer larger than $\frac{1}{2}\dim (G).$

This paper is organised as follows. In Section \ref{Sect2} we present some preliminaries on  sub-Laplacians and the Fourier analysis on compact Lie groups. In Section \ref{NikolskiiSect},  we establish a  subelliptic Nikolskii's Inequality on compact Lie groups and consequently, in Section \ref{SubBesovSpacesO} we use it to deduce non-trivial embedding properties between subelliptic Besov spaces. In Section \ref{Sect5} we will use the global functional calculus and the $L^p$-mapping properties of pseudo-differential operators to deduce continuous embeddings between subelliptic Sobolev spaces (resp. subelliptic Besov spaces) and the Sobolev spaces (resp. Besov spaces) associated to the Laplacian. In Section \ref{Triebel}, we introduce the subelliptic Triebel-Lizorkin spaces on compact Lie groups, their embedding properties and their interpolation properties in relation with subelliptic Besov spaces. Finally, with the analysis developed, we study the boundedness of pseudo-differential operators on subelliptic Sobolev and  Besov spaces in Section \ref{boundedness}.

\section{Sub-Laplacians and Fourier analysis on compact Lie groups}\label{Sect2}

Let us assume that $G$ is a compact Lie group and let $e=e_G$ be its identity element. For describing subelliptic Besov spaces in terms of the representation theory of compact Lie groups we will use the Fourier transform. One reason for this is that the Fourier transform encodes the discrete spectrum of sub-Laplacians.  If $G$ is a compact Lie group we denote by $dx\equiv d\mu(x)$  its unique (normalised)  Haar measure.  We will write $L^p(G),$ $1\leq p\leq \infty,$ for the Lebesgue spaces $L^p(G,dx).$ The unitary dual $\widehat{G}$ of $G,$  consists of the equivalence classes of continuous irreducible unitary representations of $G.$ Moreover, for every equivalence class  $[\xi]\in \widehat{G},$ there exists a  unitary matrix-representation  $\phi_\xi\in [\xi]=[\phi_\xi],$ that is, we have a homomorphism $\phi_\xi=(\phi_{\xi,i,j})_{{i,j=1}}^{{d_\xi}}$
from $G$ into $U(d_\xi),$ where $\phi_{\xi,i,j}$ are continuous functions. So, we always understand  a representation $\xi=(\xi_{i,j})_{i,j=1}^{d_\xi}$ as a unitary matrix-representation. The result that links the representation theory of compact Lie groups with the Hilbert space $L^2(G)$ is the Peter-Weyl theorem, which asserts that the set
\begin{equation}\label{Peterset}
    B=\{\sqrt{d_\xi}\xi_{i,j}:1\leq i,j\leq d_\xi,\,[\xi]\in \widehat{G}\},
\end{equation} is an orthonormal basis of $L^2(G).$ So, every function $f\in L^2(G)$ admits an expansion of the form,
\begin{equation}\label{FIF}
    f(x)=\sum_{[\xi]\in \widehat{G}}d_\xi \textnormal{
    Tr}[\xi(x)\widehat{f}(\xi)],\,\,\textnormal{a.e. }x\in G.
\end{equation}
Here, $\textnormal{Tr}(\cdot)$ is the usual  trace on matrices and, for every unitary irreducible representation $\xi,$ $\widehat{f}(\xi)$ is the matrix-valued  Fourier transform of $f$ at  $\xi$:
\begin{equation}
    \widehat{f}(\xi):=\int\limits_{G}f(x)\xi(x)^*dx\in \mathbb{C}^{d_\xi\times d_\xi}.
\end{equation} Taking into account the Fourier inversion formula \eqref{FIF}, we have the Plancherel theorem
\begin{equation}
    \Vert f\Vert_{L^2(G)}=\left( \sum_{ [\xi] \in \widehat{G}}d_\xi \Vert \widehat{f}(\xi)\Vert_{\textnormal{HS}}^2 \right)^\frac{1}{2},
\end{equation}
where $\Vert A\Vert_{\textnormal{HS}}=(A^*A)^\frac{1}{2}$ is the Hilbert-Schmidt norm of matrices. Elements in the basis $B$ are just the eigenfunctions of the positive Laplacian (defined as the Casimir element) $\mathcal{L}_G$ on $G.$ This means that, for every $[\xi]\in \widehat{G},$ there exists $\lambda_{[\xi]}\geq 0$ satisfying
\begin{equation}
    \mathcal{L}_G\xi_{i,j}=\lambda_{[\xi]}\xi_{i,j},\,\,1\leq i,j\leq d_\xi.
\end{equation}
We refer the reader  to \cite[Chapter 7]{Ruz}, for the construction of the Laplacian trough  the Killing bilinear form and for the general aspects of the representation theory of compact Lie groups.

On the other hand, 
if $X_{1},X_2,\cdots,X_k$ are vector fields satisfying the H\"ormander condition of order $\kappa,$ (see H\"ormander \cite{Hormander1967}) \begin{equation}\label{Lie}\textnormal{Lie}\{X_j:1\leq j\leq k\}=T_eG,
\end{equation} the (positive hipoellyptic) sub-Laplacian associated to the $X_i$'s is defined by
\begin{equation}\label{sublaplacian}
    \mathcal{L}_{\textnormal{sub}}:=-(X_1^2+\cdots +X_k^2).
\end{equation} 
The condition \eqref{Lie} is in turn equivalent to saying that at  every point $g\in G,$ the vector fields $X_{i}$ and the commutators
\begin{equation}
    [X_{j_1},X_{j_2}],[X_{j_1},[X_{j_2},X_{j_3}]],\cdots, [X_{j_1},[X_{j_2}, [X_{j_3},\cdots, X_{j_\tau}] ] ],\,\,1\leq \tau\leq \kappa,
    \end{equation}
generate the tangent space $T_gG$ of $G$ at $g.$ A central notion in our work is that of the Hausdorff dimension, in this case, associated to the sub-Laplacian. Indeed, for all $x\in G,$ denote by $H_{x}^\omega G$ the subspace of the tangent space $T_xG$ generated by the $X_i$'s and all the Lie brackets  $$ [X_{j_1},X_{j_2}],[X_{j_1}.[X_{j_2},X_{j_3}]],\cdots, [X_{j_1},[X_{j_2}, [X_{j_3},\cdots, X_{j_\omega}] ] ],$$ with $\omega\leq \kappa.$ The H\"ormander condition can be stated as $H_{x}^\kappa G=T_xG,$ $x\in G.$ We have the filtration
\begin{equation}
H_{x}^1G\subset H_{x}^2G \subset H_{x}^3G\subset \cdots \subset H_{x}^{\kappa-1}G\subset H_{x}^\kappa G= T_xG,\,\,x\in G.
\end{equation} In our case,  the dimension of every $H_x^\omega G$ does not depend on $x$ and we write $\dim H^\omega G:=\dim_{x}H^\omega G$ for any $x\in G.$ So, the Hausdorff dimension can be defined as (see e.g. \cite[p. 6]{HassaKoka}),
\begin{equation}\label{Hausdorff-dimension}
    Q:=\dim(H^1G)+\sum_{i=1}^{\kappa-1} (i+1)(\dim H^{i+1}G-\dim H^{i}G ).
\end{equation}

Because the symmetric operator $\mathcal{L}_{\textnormal{sub}}$ acting on $C^\infty(G)$ admits a self-adjoint extension on $L^2(G),$ that we also denote by $\mathcal{L}_{\textnormal{sub}},$  functions of the sub-Laplacian $\mathcal{L}_{\textnormal{sub}}$ can be defined according to the spectral theorem by
\begin{equation}
f(\mathcal{L}_{\textnormal{sub}})=\int\limits_{0}^\infty f(\lambda)dE_\lambda,
\end{equation}
for every measurable function $f$ defined  on $\mathbb{R}^+_0:=[0,\infty).$ Here $\{dE_\lambda\}_{\lambda>0},$ is the spectral measure associated to the spectral resolution  of $\mathcal{L}_{\textnormal{sub}}.$  In this case, 
\begin{equation}
    \textnormal{Dom}(f(\mathcal{L}_{\textnormal{sub}}))=\left\{ f\in L^2(G):  \int\limits_{0}^\infty |f(\lambda)|^2d \Vert E_\lambda f\Vert^2_{L^2(G)} <\infty   \right\},
\end{equation} where $d \Vert E_\lambda f\Vert^2_{L^2(G)}=d(E_\lambda f,f)_{L^2(G)}$ is the Riemann-Stieltjes measure induced by the spectral  resolution $\{E_\lambda\}_{\lambda>0}.$  In our further analysis the functions $\tilde{f}_{\alpha}(t)=t^{-\alpha}$ and $f_\alpha(t)=(1+t)^{-\alpha}$ will be useful,  defining  the operators
\begin{equation}
   (\mathcal{L}_{\textnormal{sub}})^{-\alpha}:=\tilde{f}_{\alpha}(\mathcal{L}_{\textnormal{sub}}), \,\,\,(1+\mathcal{L}_{\textnormal{sub}})^{-\alpha}:={f}_{\alpha}(\mathcal{L}_{\textnormal{sub}}),
\end{equation}
which are the homogeneous and inhomogeneous subelliptic  Bessel potentials of order $\alpha.$

\section{Subelliptic Nikolskii's Inequality on compact Lie groups }\label{NikolskiiSect}

Our main tool for obtaining embedding properties for Besov spaces in the subelliptic context is a suitable version of the Nikolskii's inequality. To present a subelliptic Nikolskii's inequality, let us use the usual nomenclature. If $t_{\tilde\ell}$ is a compactly supported function on $[0,\tilde\ell],$ for $\tilde\ell>0,$  we can define the operator \begin{equation}T_{\tilde\ell}:=t_{\tilde\ell}((1+\mathcal{L}_{\textnormal{sub}})^\frac{1}{2}):C^\infty(G)\rightarrow C^\infty(G), \end{equation} by the functional calculus. An explicit representation for $T_{\tilde\ell}$ can be given by
\begin{equation}\label{functionofL}
    T_{\tilde\ell} f(x)=\sum_{ [\xi]\in \widehat{G}   }d_\xi \textnormal{Tr}[\xi(x) t_{\tilde\ell}((I_{d_\xi}+\widehat{ \mathcal{L}}_{\textnormal{sub}}(\xi))^\frac{1}{2})\widehat{f}(\xi) ],\,\,f\in C^\infty(G),
\end{equation} where $\widehat{ \mathcal{L}}_{\textnormal{sub}}(\xi)=\widehat{k}_{\textnormal{sub}}(\xi)$ is the Fourier transform of the right convolution kernel $k_{\textnormal{sub}}$ of the sub-Laplacian $\mathcal{L}_{\textnormal{sub}}.$ In general, we write $t_\ell(A)$ applied to continuous operators $A$ on $C^\infty(G),$ but we also write $t_\ell(\widehat{A}(\xi))$ applied to $\widehat{A}(\xi),$ for its symbol $\widehat{t_\ell(A)}(\xi).$
So, we present our subelliptic Nikolskii's inequality. Our starting point is the following lemma proved in \cite[p. 52]{AkylzhanovRuzhansky2015}.
\begin{lemma}\label{RuzAkhy}  If $1<p\leq 2\leq q<\infty,$  and $f\in C^\infty(G),$ the estimate 
\begin{align}
    \Vert f\Vert_{{L^q}(G)}\leq C \Vert (1+\mathcal{L}_\textnormal{sub})^{  \frac{\gamma}{2}  }f\Vert_{{L^p}(G)},
\end{align}
holds true for all $\gamma\geq Q(1/p-1/q).$ 
\end{lemma}

Here $Q$ is the Hausdorff dimension of $G$ associated to the Carnot-Caratheodory  distance $\rho$ associated to $\mathcal{L}_{ \textnormal{sub} }$ (see e.g. \cite[p. 6]{HassaKoka}), which we have defined in \eqref{Hausdorff-dimension}. 

\begin{theorem}\label{Nikolskii} Let $G$ be a compact Lie group and let $t_{\tilde\ell}$  be a measurable function compactly supported in $[\frac{1}{2}\tilde\ell,\tilde\ell],$ $\tilde\ell>0.$ If $T_{\tilde\ell}$ is the operator defined by \eqref{functionofL}, then for all $1< p\leq q<\infty,$ we have
\begin{equation}
    \Vert T_{\tilde\ell} f\Vert_{L^q(G)}\leq C{\tilde\ell}^{{Q}(\frac{1}{p}-\frac{1}{q})}\Vert T_{\tilde\ell} f\Vert_{L^p(G)},\,\,f\in C^\infty(G),
\end{equation}
where $C=C_{p,q,Q}$ is independent on $f$ and $\tilde\ell.$
\end{theorem}
\begin{proof} Note that the spectrum of $(1+\mathcal{L}_{\textnormal{sub}})^\frac{1}{2}$ lies in $[1,\infty),$ so that for $0\leq \tilde\ell<1,$ $$   [\tilde\ell/2,\tilde\ell]\cap \textnormal{Spec}[ (I_{d_\xi}+\widehat{ \mathcal{L}}_{\textnormal{sub}}(\xi))^\frac{1}{2}   ]=\emptyset.$$
Consequently, for $0\leq \tilde\ell < 1,$
$$ T_{\tilde\ell}=\int\limits_{1}^{\infty}t_{\tilde\ell}((1+\lambda)^{\frac{1}{2}})dE_\lambda \equiv 0,  $$
is the null operator. In this case we have nothing to prove. So, we will assume that $\tilde\ell\geq 1.$  If $1<p\leq 2\leq q<\infty,$  Lemma \ref{RuzAkhy} gives
\begin{align}
    \Vert T_{\tilde\ell} f\Vert_{{L^q}(G)}\leq C \Vert (1+\mathcal{L}_\textnormal{sub})^{  \gamma/2  }T_{\tilde\ell} f\Vert_{{L^p}(G)},
\end{align}
for $\gamma\geq Q(1/p-1/q).$ In particular, if  $\gamma= Q(1/p-1/q),$ and $\{dE^{\mathcal{M}}_{\lambda}\}_{1\leq \lambda<\infty}$ is the spectral measure associated with $\mathcal{M}=(1+\mathcal{L}_\textnormal{sub})^{  1/2  },$ by the functional calculus we can estimate
\begin{align*}
    \Vert (1+\mathcal{L}_\textnormal{sub})^{  \gamma/2  }T_{\tilde\ell} f\Vert_{{L^p}(G)}&=\left\Vert       \int\limits_{0}^{\infty}  \lambda^{\gamma}  t_{\tilde\ell}(\lambda) dE^{\mathcal{M}}_\lambda  f  \right\Vert_{{L^p}(G)}\\
    &=\tilde\ell^{\gamma }\left\Vert       \int\limits_{    \frac{\tilde\ell}{2} }^{\tilde\ell}  \tilde\ell^{-\gamma }\lambda^{\gamma} t_{\tilde\ell}(\lambda) dE^{\mathcal{M}}_\lambda  f  \right\Vert_{{L^p}(G)}\\
    &\asymp \tilde\ell^{\gamma } \left\Vert       \int\limits_{   \frac{\tilde\ell}{2} }^{\tilde\ell}  t_{\tilde\ell}(\lambda)  dE^{\mathcal{M}}_\lambda  f  \right\Vert_{{L^p}(G)} .
\end{align*} Consequently
\begin{align}
    \Vert  (1+\mathcal{L}_\textnormal{sub})^{  \gamma/2  }T_{\tilde\ell} f   \Vert_{{L^p}(G)}\lesssim  {\tilde\ell}^{\gamma} \left\Vert       \int\limits_{   \frac{\tilde\ell}{2} }^{\tilde\ell} t_{\tilde\ell}(\lambda)   dE_\lambda  f  \right\Vert_{{L^p}(G)}= {\tilde\ell}^{\gamma} \Vert T_{\tilde\ell} f \Vert_{{L^p}(G)},
\end{align} 

So, we obtain
\begin{equation}
    \Vert T_{\tilde\ell} f\Vert_{{L^q}(G)}\leq C {\tilde\ell}^{Q(\frac{1}{p}-\frac{1}{q}) } \Vert T_{\tilde\ell} f \Vert_{{L^p}(G)},\,\,1<p\leq 2\leq q<\infty.
\end{equation}

Now, we generalise this result to the complete range  $1<p\leq q\leq 2$ by using interpolation inequalities. In this case  from the inequality $1<p\leq q\leq 2\leq q'<\infty,$ we get  the following estimates,
\begin{equation}
    \Vert T_{\tilde\ell} f\Vert_{{L^{q'}}(G)}\leq C {\tilde\ell}^{Q(\frac{1}{p}-\frac{1}{q'})}\left\Vert        T_{\tilde\ell} f  \right\Vert_{{L^p}(G)},
\end{equation} and 
\begin{equation}
    \Vert T_{\tilde\ell} f\Vert_{{L^{q}}(G)}\leq \Vert        T_{\tilde\ell} f  \Vert_{{L^p}(G)}^\theta  \Vert        T_{\tilde\ell} f  \Vert_{{L^{q'}}(G)}^{1-\theta},\,\,1/q=\theta/p+(1-\theta)/q',
\end{equation} where we have used the interpolation inequality  (see e.g. Brezis \cite[Chapter 4]{Brezis})
\begin{equation}
    \Vert h\Vert_{{L^{r}}(X,\mu)}\leq \Vert        h \Vert_{{L^{r_0}}(X,\mu)}^\theta  \Vert        h  \Vert_{{L^{r_1}}(X,\mu)}^{1-\theta},\,\,1/r=\theta/r_0+(1-\theta)/r_1,\,\,1\leq r_0<r_1\leq \infty.
\end{equation}
So, we have
\begin{align*}
     \Vert T_{\tilde\ell} f\Vert_{{L^{q}}(G)}\lesssim  \Vert        T_{\tilde\ell} f  \Vert_{{L^p}(G)}^\theta {\tilde\ell}^{{Q(1-\theta)}(\frac{1}{p}-\frac{1}{q'})}\left\Vert        T_{\tilde\ell }f  \right\Vert_{{L^p}(G)}^{1-\theta} =   {\tilde\ell}^{{Q(1-\theta)}(\frac{1}{p}-\frac{1}{q'})}\Vert        T_{\tilde\ell} f  \Vert_{{L^p}(G)}. 
\end{align*} Since  $\theta$ satisfies $(1-\theta)/q'=1/q-\theta/p,$ from the identity
$$  {\tilde\ell}^{{Q(1-\theta)}(\frac{1}{p}-\frac{1}{q'})}=  {\tilde\ell}^{{Q}(\frac{1-\theta}{p}-\frac{1-\theta}{q'})}=  {\tilde\ell}^{{Q}(\frac{1-\theta}{p}-\frac{1}{q}+\frac{\theta}{p})}= {\tilde\ell}^{{Q}(\frac{1}{p}-\frac{1}{q})},$$ we complete the proof for this case. Now, if $2\leq p\leq q<\infty,$ the proof will be based in the action of the linear operator $T_{\tilde\ell}$ on the Banach spaces
\begin{equation}
    X_{s,{\tilde\ell}}(G):=\{f\in L^s(G): T_{\tilde\ell} f=f\},\,\, \tilde\ell\geq 1,\,\,1<s<\infty,
\end{equation} endowed with the norm
\begin{equation}
    \Vert f \Vert_{X_{s,\tilde\ell}(G)}:=\Vert f \Vert_{L^s(G)} =\left(\int\limits_{G}|f(x)|^sdx\right)^{\frac{1}{s}} .
\end{equation}
    
Putting $s=p'$ and $r=q'$ and taking into account that the parameters $r,s$ satisfy the inequality, $1<r\leq s\leq 2,$ we have the estimate
\begin{equation}
     \Vert f \Vert_{X_{s,\tilde\ell}(G)}=\Vert T_{\tilde\ell} f\Vert_{{L^s}(G)}\leq C{\tilde \ell}^{{Q}(\frac{1}{r}-\frac{1}{s}) } \Vert T_{\tilde\ell} f \Vert_{{L^r}(G)}=C {\tilde\ell}^{{Q}(\frac{1}{r}-\frac{1}{s}) } \Vert  f \Vert_{ X_{r,\tilde\ell}(G)   }.
\end{equation} Consequently,  $T_{\tilde\ell}:X_{r,\tilde\ell}(G) \rightarrow X_{s,\tilde\ell}(G)$ extends to a bounded operator with the operator norm satisfying 
\begin{equation}
    \Vert T_{\tilde\ell}\Vert_{\mathscr{B}(X_{r,\tilde\ell}(G) , X_{s,\tilde\ell}(G))}\leq C {\tilde\ell}^{{Q}(\frac{1}{r}-\frac{1}{s}) }.
\end{equation} By using the identification, $(X_{r,\tilde\ell}(G))^*=X_{q,\tilde\ell}(G)$ and $(X_{s,\tilde\ell}(G))^*=X_{p,\tilde\ell}(G)$ where $E^*$ denotes the dual of a Banach space $E,$ the argument of duality gives the boundedness of $T_{\tilde\ell}$ from $X_{p,\tilde\ell}(G)$ into $X_{q,\tilde\ell}(G)$ with operator norm satisfying, \begin{equation}
    \Vert T_{\tilde\ell}\Vert_{\mathscr{B}(X_{q,\tilde\ell}(G) , X_{p,\tilde\ell}(G))}\leq C {\tilde\ell}^{{Q}(\frac{1}{r}-\frac{1}{s}) }=C{\tilde\ell}^{{Q}(\frac{1}{q'}-\frac{1}{p'}) }=C{\tilde\ell}^{{Q}(\frac{1}{p}-\frac{1}{q}) }.
\end{equation} Thus, we obtain the estimate,
\begin{equation}
    \Vert T_{\tilde\ell} f\Vert_{{L^q}(G)}\leq C {\tilde\ell}^{{Q}(\frac{1}{p}-\frac{1}{q}) } \Vert T_{\tilde\ell} f \Vert_{{L^p}(G)},\,\,2\leq p\leq q <\infty.
\end{equation} So, we finish the proof.
\end{proof}

\begin{corollary} Let $G$ be a compact Lie group and let $t_{\tilde\ell}$  be a measurable function compactly supported in $[\frac{1}{2}\tilde\ell,\tilde\ell],$ $\tilde\ell>0.$ If $T_{\tilde\ell}$ is the operator defined by \eqref{functionofL}, then $T_{\tilde\ell}$ from $X_{p,\tilde\ell}(G)$ into $X_{q,\tilde\ell}(G)$ extends to a bounded operator, with the operator norm satisfying \begin{equation}
    \Vert T_{\tilde\ell}\Vert_{\mathscr{B}(X_{q,\tilde\ell}(G) , X_{p,\tilde\ell}(G))}\leq C{\tilde\ell}^{{Q}(\frac{1}{p}-\frac{1}{q}) },
\end{equation}  for all $1\leq p\leq q<\infty.$ Moreover, for all $1\leq p< \infty,$ $T_{\tilde\ell}$ from $X_{p,\tilde\ell}(G)$ into $X_{p,\tilde\ell}(G)$ extends to a bounded operator, with the operator norm satisfying, \begin{equation}
    \Vert T_{\tilde\ell}\Vert_{\mathscr{B}(X_{p,\tilde\ell}(G) , X_{p,\tilde\ell}(G))}\leq C.
\end{equation}

\end{corollary}

\section{Subelliptic Besov spaces}\label{SubBesovSpacesO}

In this section we introduce and  examine some embedding properties for those Besov spaces associated to sub-Laplacians on compact Lie groups which we call subelliptic Besov spaces. Troughout this section, $X=\{X_1,\cdots,X_k\}$ is a system of vector fields satisfying the H\"ormander condition of order $\kappa\in\mathbb{N},$ and by simplicity throughout this work  we will use the following notation $\mathcal{L}\equiv \mathcal{L}_{\textnormal{sub}}=-\sum_{j=1}^k X_j^2$ for  the associated sub-Laplacian. This condition guarantees the hypoellipticity of $\mathcal{L}$ (see H\"ormander \cite{Hormander1967}).

\subsection{Motivation and definition of subelliptic Besov spaces}
The main tool in the formulation of the Besov spaces is the notion of dyadic decompositions. So, if $\psi_{\tilde\ell}$ is the characteristic function of the interval $I_{\tilde\ell}:=[2^{\tilde\ell},2^{\tilde\ell+1}),$ for $\tilde\ell\in \mathbb{N}_0,$  we define the operator
\begin{equation}\label{functionofL''}
    \psi_{\tilde\ell}(D) f(x):=\sum_{ [\xi]\in \widehat{G}   }d_\xi \textnormal{Tr}[\xi(x) \psi_{\tilde\ell}((I_{d_\xi}+\widehat{ \mathcal{L}}(\xi))^\frac{1}{2})\widehat{f}(\xi) ],\,\,f\in C^\infty(G).
\end{equation} If we denote by $\{E_\lambda\}_{0\leq \lambda<\infty},$ $E_\lambda:=E_{[0,\lambda)},$ the spectral measure associated to the subelliptic Bessel potential $B_{ \mathcal{L} }=(1+\mathcal{L})^\frac{1}{2},$ we have
$
   \psi_{\tilde\ell}(D)=E_{2^{\tilde\ell+1}} -E_{2^{\tilde\ell}}.
$
From the orthogonality of the spectral projections, for every $f\in L^2(G)$ we have,
\begin{equation}
    \Vert f\Vert_{L^2(G)}=\left\Vert \sum_{\tilde\ell=0}^\infty \psi_{\tilde\ell}(D)f\right\Vert_{L^2(G)}=\left( \sum_{\tilde\ell=0}^\infty\Vert \psi_{\tilde\ell}(D)f \Vert_{L^2(G)}^2 \right)^{\frac{1}{2}}.
\end{equation}
This simple analysis shows that we can measure the square-integrability of a function $f$ using the family of operators $\psi_{\tilde\ell}(D),$ $\tilde\ell\in\mathbb{N}_0.$ To measure the regularity of functions/distributions, it is usual to use  Sobolev spaces, in our case, associated to sub-Laplacians. But, if we want to measure at the same time, the regularity and the integrability of functions/distributions, Besov spaces are the appearing spaces.  By following \cite{GarettoRuzhansky2015}, the subelliptic Sobolev space $H^{s,\mathcal{L}}(G)$ of order $s\in \mathbb{R},$ is defined by the condition
\begin{equation}\label{subellipticSobolevspaxce}
f\in H^{s,\mathcal{L}}(G) \textnormal{  if and only if }\Vert f \Vert_{ H^{s,\mathcal{L}}(G)   }:=\Vert (1+\mathcal{L})^\frac{s}{2} f\Vert_{L^2(G)}<\infty.
\end{equation}
So, 
in terms of the family $\psi_{\tilde\ell}(D),$ $\tilde\ell\in\mathbb{N}_0,$  we can write,
\begin{equation}
 \Vert f \Vert_{ H^{s,\mathcal{L}}(G)   }=\left\Vert \sum_{\tilde\ell=0}^\infty \psi_{\tilde\ell}(D)(1+\mathcal{L})^\frac{s}{2}f \right\Vert_{L^2(G)}=\left( \sum_{\tilde\ell=0}^\infty\Vert \psi_{\tilde\ell}(D)(1+\mathcal{L})^\frac{s}{2}f \Vert_{L^2(G)}^2 \right)^{\frac{1}{2}}.    
\end{equation} This allows us to write
\begin{align*}
    \Vert \psi_{\tilde\ell}(D)(1+\mathcal{L})^\frac{s}{2}f \Vert_{L^2(G)} &:=\Vert [E_{2^{\tilde\ell+1}} -E_{2^{\tilde\ell}}](1+\mathcal{L})^\frac{s}{2}f \Vert_{L^2(G)}\\
    &=\left\Vert \int_{2^{\tilde\ell}}^{2^{\tilde\ell+1}}(1+\lambda)^s dE_\lambda f \right\Vert_{L^2(G)}\\
    &= 2^{s\tilde\ell}\left\Vert \int_{2^{\tilde\ell}}^{2^{\tilde\ell+1}}(2^{-\tilde\ell}(1+\lambda))^s dE_\lambda f \right\Vert_{L^2(G)}\\
    &\asymp  2^{s\tilde\ell}\left\Vert \int_{2^{\tilde\ell}}^{2^{\tilde\ell+1}} dE_\lambda f \right\Vert_{L^2(G)}\\
    &=2^{s\tilde\ell}\Vert \psi_{\tilde\ell}(D)f \Vert_{L^2(G)}.
\end{align*} So, we obtain the following estimate
  \begin{equation}\label{HsBesov}
 \Vert f \Vert_{ H^{s,\mathcal{L}}(G)   }\asymp \left( \sum_{\tilde\ell=0}^\infty  2^{2\tilde\ell s}\Vert \psi_{\tilde\ell}(D)f \Vert_{L^2(G)}^2 \right)^{\frac{1}{2}}=:\Vert f\Vert_{B^{s,\mathcal{L}}_{2,2}(G)}.    
\end{equation}  We then deduce that the norm appearing in  the right hand side of \eqref{HsBesov} defines an equivalent norm for the Sobolev space $H^{s,\mathcal{L}}(G).$ With this discussion in mind, Besov spaces can be introduced by the varying of  the parameters measuring the integrability and the regularity of functions/distributions in \eqref{HsBesov}. So, for $s\in\mathbb{R},$ $0<q<\infty,$ the subelliptic Besov space $B^{s,\mathcal{L}}_{p,q}(G)$ consists of those functions/distributions satisfying
\begin{equation}
   \Vert f\Vert_{ B^{s,\mathcal{L}}_{p,q}(G)  } =\left( \sum_{\tilde\ell=0}^\infty  2^{\tilde\ell q s}\Vert \psi_{\tilde\ell}(D)f \Vert_{L^p(G)}^q \right)^{\frac{1}{q}}<\infty,
\end{equation} for $0<p\leq \infty,$ with the following modification
\begin{equation}
   \Vert f\Vert_{ B^{s,\mathcal{L}}_{p,\infty}(G)  } = \sup_{\tilde\ell\in\mathbb{N}_0}  2^{\tilde\ell  s}\Vert \psi_{\tilde\ell}(D)f \Vert_{L^p(G)} <\infty,
\end{equation} when $q=\infty.$
    Clearly, from the earlier discussion, for every $s\in \mathbb{R},$  $H^{s,\mathcal{L}}(G)=B^{s,\mathcal{L}}_{2,2}(G),$ and particularly, $L^{2}(G)=B^{0,\mathcal{L}}_{2,2}(G).$ Sobolev spaces modelled on $L^p$-spaces and associated to the sub-Laplacian will be denoted by $L^{p,\mathcal{L}}_r(G),$ and defined by the norm
    \begin{equation}
\Vert f \Vert_{{L}_r^{p,\mathcal{L}}}:=\Vert (1+\mathcal{L})^\frac{r}{2} f\Vert_{L^{p}}.
\end{equation} Often, we will show that Besov spaces $B^{s,\mathcal{L}}_{p,q}$ are intermediate spaces between the Sobolev spaces $L^{p,\mathcal{L}}_r(G).$ So, we will show that Besov spaces can be obtained from the interpolation of Sobolev spaces.
\begin{remark}[Besov spaces associated to the Laplacian] Let $G$ be a compact Lie group. The  positive Laplacian  on $G$  can be defined as 
\begin{equation}\label{laplacianonG}
    \mathcal{L}_{{G}}:=-(X_1^2+\cdots +X_n^2),\,\,n=\dim(G),
\end{equation} where $X=\{X_{i}:1\leq i\leq n\}$ is a basis for the Lie algebra $\mathfrak{g}=\textnormal{Lie}(G)$ of $G.$ The global symbol associated to the Laplacian is given by 
\begin{equation}
    \widehat{\mathcal{L}}_{{G}}(\xi)=\lambda_{[\xi]}I_{d_\xi},
\end{equation} where $\{\lambda_{[\xi]}:[\xi]\in \widehat{G}  \}$ is the spectrum of the Laplacian. This sequence has the property that $\mathcal{L}_{G}\xi_{ij}=\lambda_{[\xi]}\xi_{ij},$ $1\leq i,j\leq d_\xi,$ where $d_\xi$ is the dimension of the representation $\xi,$ for every $[\xi]\in \widehat{G}.$ This means that  $\lambda_{[\xi]}$ is an eigenvalue of the Laplacian, with the corresponding eigenspace $E_\xi:=\{\xi_{ij}(x) :1\leq i,j\leq d_\xi \}.$ Also, we note that the multiplicity of the eigenvalue $\lambda_{[\xi]}$ is $d_\xi^2.$ Putting $\langle \xi\rangle:=(1+\lambda_{[\xi]})^{\frac{1}{2}},$ we have
$
    (I_\xi+\widehat{\mathcal{L}}_{{G}}(\xi))^{\frac{1}{2}}=\langle \xi\rangle I_{d_\xi}.
$ Consequently, for every $[\xi]\in \widehat{G}$ and $\tilde\ell\in \mathbb{N}_{0},$ the symbol of the operator
$
   {\psi_{\tilde\ell}((I+{\mathcal{L}}_{{G}})^{\frac{1}{2}})}
$ is $\psi_{\tilde\ell}(\langle \xi\rangle)I_{d_\xi}. $ This argument implies that the sequence of operators $\{\psi_{\tilde\ell}(D)  \}_{\tilde\ell}$ is defined by
\begin{equation}\label{functionofLaplacian}
    \psi_{\tilde\ell}((1+\mathcal{L}_G)^\frac{1}{2}) f(x):=\sum_{ 2^{\tilde\ell}\leq \langle \xi\rangle<2^{\tilde\ell+1}   }d_\xi \textnormal{Tr}[\xi(x) \widehat{f}(\xi) ],\,\,f\in C^\infty(G).
\end{equation} We conclude that the Besov spaces on compact Lie groups associated to the Laplacian are defined by 
\begin{equation}\label{q<inftyLaP'}
   \Vert f\Vert_{ B^{s,\mathcal{L}_G}_{p,q}(G)  } =\left( \sum_{\tilde\ell=0}^\infty  2^{\tilde\ell q s}\left\Vert  \sum_{ 2^{\tilde\ell}\leq \langle \xi\rangle<2^{\tilde\ell+1}   }d_\xi \textnormal{Tr}[\xi(x) \widehat{f}(\xi) ]\right\Vert_{L^p(G)}^q \right)^{\frac{1}{q}}<\infty,
\end{equation} for $0<p\leq \infty,$ and
\begin{equation}\label{q=inftyLaP'}
   \Vert f\Vert_{ B^{s,\mathcal{L}_G}_{p,\infty}(G)  } = \sup_{\tilde\ell\in\mathbb{N}_0}  2^{\tilde\ell  s}\left\Vert  \sum_{ 2^{\tilde\ell}\leq \langle \xi\rangle<2^{\tilde\ell+1}   }d_\xi \textnormal{Tr}[\xi(x) \widehat{f}(\xi) ]\right\Vert_{L^p(G)} <\infty,
\end{equation} for $q=\infty.$ The definitions \eqref{q<inftyLaP'} and \eqref{q=inftyLaP'} were introduced by E. Nursultanov, S. Tikhonov and the second author in \cite{NurRuzTikhBesov2015} and consistently developed on arbitrary compact homogeneous manifolds in \cite{NurRuzTikhBesov2017}. 
\end{remark}

\begin{remark}[Fourier description for subelliptic Besov spaces]
Now, we discuss the formulation of subelliptic Besov spaces in terms of the Fourier analysis associated to the sub-Laplacian $\mathcal{L}=-\sum_{j=1}^k X_j^2.$ The self-adjointness of this operator, implies that the global symbol $\widehat{\mathcal{L}}(\xi)$ is symmetric and it can be assumed diagonal for every $[\xi]\in \widehat{G},$ under a suitable choose of the basis in the representation spaces. So, there exists a non-negative sequence $\{\nu_{ii}\}_{i=1}^{d_\xi},$ such that
\begin{equation}
\widehat{\mathcal{L}}(\xi):=(\mathcal{L}\xi)(e_G)   =\begin{bmatrix}
    \nu_{11}(\xi)^2 & 0 & 0 & \dots  & 0 \\
     0 & \nu_{22}(\xi)^2& 0 & \dots  & 0 \\
    \vdots & \vdots & \vdots & \ddots & \vdots \\
   0 & 0 &0 & \dots  & \nu_{d_\xi d_\xi}(\xi)^2
\end{bmatrix}=:\textnormal{diag}[\nu_{ii}(\xi)^2]_{1\leq i\leq d_\xi}.
\end{equation}
So, the symbol $\widehat{M}(\xi):=(M\xi)(e_G)$ of the operator $(1+\mathcal{L})^{\frac{1}{2}}$ is given by,
\begin{equation}
\widehat{M}(\xi)  =\begin{bmatrix}
    (1+\nu_{11}(\xi)^2)^{\frac{1}{2}} & 0 & 0 & \dots  & 0 \\
     0 & (1+\nu_{22}(\xi)^2)^{\frac{1}{2}}  & 0 & \dots  & 0 \\
    \vdots & \vdots & \vdots & \ddots & \vdots \\
   0 & 0 &0 & \dots  & (1+\nu_{d_\xi d_\xi}(\xi)^2)^{\frac{1}{2}}
\end{bmatrix}.
\end{equation} We  write  $\widehat{M}(\xi)= \textnormal{diag}[(1+\nu_{ii}(\xi)^2)^{\frac{1}{2}}]_{1\leq i\leq d_\xi} .  $   If  $\psi_{\tilde\ell}(\xi)$ denotes the symbol of the operator $\psi_{\tilde\ell}(D),$ then we have
\begin{equation}
    \psi_{\tilde\ell}(\xi)=\textnormal{diag}[\psi_{\tilde\ell}((1+\nu_{ii}(\xi)^2)^{\frac{1}{2}})]_{1\leq i\leq d_\xi},\,\,\tilde\ell\in\mathbb{N}_0,\,\,[\xi]\in \widehat{G}.
\end{equation} We conclude that the analogues of  definitions \eqref{q<inftyLaP'} and \eqref{q=inftyLaP'} for subelliptic Besov spaces can be re-written as
\begin{equation}\label{q<inftyLaP}
   \Vert f\Vert_{ B^{s,\mathcal{L}}_{p,q}(G)  }^q =\sum_{\tilde\ell=0}^\infty  2^{\tilde\ell q s}\left\Vert  \sum_{ [\xi]\in \widehat{G}}d_\xi \textnormal{Tr}[\xi(x)\textnormal{diag}[\psi_{\tilde\ell}((1+\nu_{ii}(\xi)^2)^{\frac{1}{2}})] \widehat{f}(\xi) ]\right\Vert_{L^p(G)}^q <\infty,
\end{equation} for $0<p\leq \infty,$ and
\begin{equation}\label{q=inftyLaP}
   \Vert f\Vert_{ B^{s,\mathcal{L}}_{p,\infty}(G)  } = \sup_{\tilde\ell\in\mathbb{N}_0}  2^{\tilde\ell  s}\left\Vert  \sum_{ [\xi]\in \widehat{G}   }d_\xi \textnormal{Tr}[\xi(x) \textnormal{diag}[\psi_{\tilde\ell}((1+\nu_{ii}(\xi)^2)^{\frac{1}{2}})]\widehat{f}(\xi) ]\right\Vert_{L^p(G)} <\infty,
\end{equation} for $q=\infty.$

\end{remark}

\subsection{Embedding properties for subelliptic Besov spaces}

In this section we study some embedding properties between Besov spaces and their interplay with Sobolev spaces modeled on $L^p$-spaces.  To accomplish this, we will use  the Littlewood-Paley theorem,
\begin{equation}\label{LitPaleyTh}
  \Vert f \Vert_{F^{0,\mathcal{L}}_{p,2}(G)    }:= \Vert   [\sum_{s\in\mathbb{N}_0}  |\psi_{s}(D)f(x)       |^{2}]^{\frac{1}{2}}\Vert_{L^{p}(G)}
\asymp \Vert f \Vert_{L^p(G)}.  \end{equation} A proof of this result can be found in Furioli, Melzi and Veneruso \cite{furioli} where the Littlewood-Paley theorem for sub-Laplacians was proved for Lie groups of polynomial growth. 

\begin{theorem}\label{Emb1}
Let $G$ be a compact Lie group and let us denote by  $Q$  the Hausdorff dimension of $G$ associated to the control distance associated to the sub-Laplacian $\mathcal{L}=-(X_1^2+\cdots +X_k^2),$ where the system of vector fields $X=\{X_i\}$ satisfies the H\"ormander condition.  Then 
\begin{itemize}
\item[(1)] ${B}^{r+\varepsilon,\mathcal{L}}_{p,q_1}(G)\hookrightarrow {B}^{r,\mathcal{L}}_{p,q_1}(G)\hookrightarrow {B}^{r,\mathcal{L}}_{p,q_2}(G)\hookrightarrow {B}^{r,\mathcal{L}}_{p,\infty}(G),$  $\varepsilon>0,$ $0<p\leq \infty,$ $0<q_{1}\leq q_2\leq \infty.$\\
\item[(2)]  ${B}^{r+\varepsilon,\mathcal{L}}_{p,q_1}(G)\hookrightarrow {B}^{r,\mathcal{L}}_{p,q_2}(G)$, $\varepsilon>0,$ $0<p\leq \infty,$ $1\leq q_2<q_1<\infty.$\\
\item[(3)]  ${B}^{r_1,\mathcal{L}}_{p_1,q}(G)\hookrightarrow {B}^{r_2,\mathcal{L}}_{p_2,q}(G),$  $1\leq p_1\leq p_2\leq \infty,$ $0<q<\infty,$ $r_{1}\in\mathbb{R},$ and $r_2=r_1- {Q}(\frac{1}{p_1}-\frac{1}{p_2}).$\\
\item[(4)] ${H}^{r,\mathcal{L}}(G)={B}^{r,\mathcal{L}}_{2,2}(G)$ and ${B}^{r,\mathcal{L}}_{p,p}(G)\hookrightarrow {L}_r^{p,\mathcal{L}}(G)\hookrightarrow {B}^{r,\mathcal{L}}_{p,2}(G),$ $1<p\leq 2.$\\
\item[(5)] ${B}^{r,\mathcal{L}}_{p,1}(G)\hookrightarrow L^{q}(G), $ $1\leq p\leq q\leq \infty,$ $r= {Q}(\frac{1}{p}-\frac{1}{q})$ and $L^{q}(G)\hookrightarrow {B}^{0,\mathcal{L}}_{q,\infty}(G)$ for $1<q\leq \infty.$
\end{itemize}
\end{theorem}
\begin{proof}
For the proof of $(1)$ let us fix $\varepsilon>0,$ $0<p\leq \infty,$ and $0<q_{1}\leq q_2\leq \infty.$ We observe that
\begin{equation}
\Vert f \Vert_{{B}^{r,\mathcal{L}}_{p,\infty}}=\sup_{s\in\mathbb{N}_0}2^{r s }\Vert \psi_{s}(D) f\Vert_{L^{p}}\leq \Vert \{ 2^{ s  r}\Vert \psi_{s}(D)f\Vert_{L^{p}} \}_{s\in\mathbb{N}_0} \Vert_{l^{q_2}(\mathbb{N}_0)}\equiv\Vert f \Vert_{{B}^{r,\mathcal{L}}_{p,q_2}}.
\end{equation}
So, we have that ${B}^{r,\mathcal{L}}_{p,q_2}(G) \hookrightarrow {B}^{r,\mathcal{L}}_{p,\infty}(G).$ From the estimates
\begin{equation}
\Vert f \Vert_{{B}^{r,\mathcal{L}}_{p,q_2}} \leq \Vert \{ 2^{ s r}\Vert \psi_{s}(D)f\Vert_{L^{p}} \}_{s\in\mathbb{N}_0} \Vert_{l^{q_1}(\mathbb{N}_0)}\equiv \Vert f \Vert_{{B}^{r,\mathcal{L}}_{p,q_1}},
\end{equation} and
\begin{equation}
 \Vert f \Vert_{{B}^{r,\mathcal{L}}_{p,q_1}}\leq \Vert \{ 2^{ s(r+\varepsilon)} \Vert \psi_{s}(D)f\Vert_{L^{p}} \}_{s\in\mathbb{N}_0} \Vert_{l^{q_1}(\mathbb{N}_0)}\equiv\Vert f \Vert_{{B}^{r+\varepsilon,\mathcal{L}}_{p,q_1}},
\end{equation} we deduce the embeddings ${B}^{r,\mathcal{L}}_{p,q_1}(G)\hookrightarrow {B}^{r,\mathcal{L}}_{p,q_2}(G)\hookrightarrow {B}^{r,\mathcal{L}}_{p,\infty}(G).$ 
For the proof of $(2)$ we consider $\varepsilon>0,$ $0<p\leq \infty,$ $1\leq q_2<q_1<\infty.$ Now, by the H\"older inequality,
\begin{align*}
\Vert f \Vert_{{B}^{r,\mathcal{L}}_{p,q_2}}&=\Vert \{ 2^{ s r}\Vert \psi_{s}(D)f\Vert_{L^{p}} \}_{s\in\mathbb{N}_0} \Vert_{l^{q_2}(\mathbb{N}_0)}=\Vert \{ 2^{  s (r+\varepsilon)-s\varepsilon}\Vert \psi_{s}(D)f\Vert_{L^{p}} \}_{s\in\mathbb{N}_0} \Vert_{l^{q_2}(\mathbb{N}_0)}\\
&\leq \Vert \{ 2^{   s(r+\varepsilon)q_2}\Vert \psi_{s}(D)f\Vert^{q_2}_{L^{p}} \}_{s\in\mathbb{N}_0} \Vert^{\frac{1}{q_2}}_{l^{\frac{q_1}{q_2}}(\mathbb{N}_0)}[\sum_{s\in\mathbb{N}_0}2^{ -\frac{ s\varepsilon q_2q_1 }{q_1-q_2  }    }]^{\frac{1}{q_2}-\frac{1}{q_1}}\\
&\lesssim \Vert f \Vert_{{B}^{r+\varepsilon,\mathcal{L}}_{p,q_1}}.
\end{align*} So, we obtain that ${B}^{r+\varepsilon,\mathcal{L}}_{p,q_1}(G)\hookrightarrow {B}^{r,\mathcal{L}}_{p,q_2}(G)$.
In order to prove $(3)$ we use  from Theorem \ref{Nikolskii}  the estimate \begin{equation}\Vert \psi_{s}(D)f\Vert_{L^{p_2}}\leq C2^{ {Q}(\frac{1}{p_1}-\frac{1}{p_2})s}\Vert \psi_{s}(D)f\Vert_{L^{p_1}},\,\,\,1\leq p_1\leq p_2<\infty,
\end{equation} 
so that we have
$$ \left(   \sum_{s\in\mathbb{N}_{0}}2^{   sr_2q}\Vert \psi_{s}(D)f\Vert^{q}_{L^{p_2}{(G)}}\right)^{\frac{1}{q}}\lesssim \left(   \sum_{s\in\mathbb{N}_{0}}2^{  s[r_{2}+ {Q}(\frac{1}{p_1}-\frac{1}{p_2}) ]q}\Vert \psi_{s}(D)f\Vert^{q}_{L^{p_1}{(G)}}\right)^{\frac{1}{q}} .$$ This estimate proves that ${B}^{r_1,\mathcal{L}}_{p_1,q}(G)\hookrightarrow {B}^{r_2,\mathcal{L}}_{p_2,q}(G),$ for $1\leq p_1\leq p_2\leq \infty,$ $0<q<\infty,$ $r_{1}\in\mathbb{R}$ and $r_2=r_1- {Q}(\frac{1}{p_1}-\frac{1}{p_2}).$
Now we will prove $(4),$ that is ${B}^{r,\mathcal{L}}_{p,p}(G)\hookrightarrow {L}_r^{p,\mathcal{L}}(G)\hookrightarrow {B}^{r,\mathcal{L}}_{p,2}(G),$ for $1<p\leq 2.$ By the definition of the subelliptic $p$-Sobolev  norm,
$
\Vert f \Vert_{{L}_r^{p,\mathcal{L}}}\equiv\Vert (1+\mathcal{L})^{\frac{r}{2}} f\Vert_{L^{p}},
$ if $\{E^{\mathcal{M}}(\lambda)\}_{0\leq \lambda<\infty}$ is the spectral measure associated to $\mathcal{M}=(1+\mathcal{L})^{\frac{1}{2}},$ we have,
\begin{align*}
\Vert f \Vert_{{L}_r^{p,\mathcal{L}}}^{p}&\equiv\Vert \mathcal{M}^rf\Vert^{p}_{L^{p}}=\Vert  \int\limits_{0}^{\infty}\lambda^{r}dE^{\mathcal{M}}(\lambda)f   \Vert_{L^p}^{p}\\
&=\Vert  \sum_{s\in \mathbb{N}_0}\int_{2^s}^{2^{s+1}} \lambda^{ r }dE^{\mathcal{M}}(\lambda)f   \Vert_{L^p}^{p}\leq   \sum_{s\in \mathbb{Z}} \Vert \int_{2^s}^{2^{s+1}} \lambda^{r  }dE^{\mathcal{M}}(\lambda)f   \Vert_{L^p}^{p}\\
&=   \sum_{s\in \mathbb{Z}}2^{srp} \Vert \int_{2^s}^{  2^{s+1}  }2^{-sr} \lambda^{\frac{r}{\nu}}dE^{\mathcal{M}}(\lambda)f   \Vert_{L^p}^{p}\\ &\asymp  \sum_{s\in \mathbb{Z}}2^{ srp} \Vert \int_{2^s}^{2^{s+1}}dE^{\mathcal{M}}(\lambda)f   \Vert_{L^p}^{p}= \sum_{s\in \mathbb{N}_0}2^{srp} \Vert \psi_{s}(D)f   \Vert_{L^p}^{p}\\
&=\Vert f\Vert^p_{{B}^{r,\mathcal{L}}_{p,p}}.
\end{align*}
For the other embedding we use the following version of the Minkowski  integral inequality 
$$ \left(\sum_{j=0}^\infty\left(\int_{X} |f_j(x)| d\mu(x)\right)^{\alpha}\right)^{\frac{1}{\alpha}}\leq\int_{X}\left(\sum_{j=0}^\infty |f_j(x)|^{\alpha}\right)^{\frac{1}{\alpha}}d\mu(x) ,\,\,\,f_j\textnormal{ measurable, }$$
with $\alpha=\frac{2}{p}.$ So, we get
\begin{align*}
\Vert f \Vert_{{B}^{r,\mathcal{L}}_{p,2}}&=\left( \sum_{s\in \mathbb{N}_0} 4^{rs}\Vert\psi_{s}(D)f \Vert^{2}_{L^p}\right)^{\frac{1}{2}}=\left( \sum_{s\in \mathbb{N}_0} 4^{rs}[\int_{G}\vert\psi_{s}(D)f(x) \vert^{p}dx]^{\frac{2}{p}}\right)^{\frac{p}{2} \cdot \frac{1}{p}}\\
&\leq\left[ \int_{G} [  \sum_{s\in\mathbb{N}_0}4^{sr}  |\psi_{s}(D)f(x)       |^{\frac{2}{p}p}dx]^{\frac{p}{2}}   \right]^{\frac{1}{p}}\\&=\left[ \int_{G} [  \sum_{s\in\mathbb{N}_0}4^{sr}  |\psi_{s}(D)f(x)       |^{2}dx]^{\frac{p}{2}}   \right]^{\frac{1}{p}}\\
&=\Vert   [\sum_{s\in\mathbb{N}_0}4^{sr}  |\psi_{s}(D)f(x)       |^{2}dx]^{\frac{1}{2}}\Vert_{L^{p}}\asymp \Vert   [\sum_{s\in\mathbb{N}_0}  |\psi_{s}(D)( (1+\mathcal{L})^\frac{r}{2} f)(x)       |^{2}dx]^{\frac{1}{2}}\Vert_{L^{p}}\\
&\asymp \Vert(1+\mathcal{L})^\frac{r}{2}f \Vert_{L^p}=\Vert f\Vert_{{H}^{r,p,\mathcal{L}}},
\end{align*}
using the Littlewood-Paley theorem (see Eq. \eqref{LitPaleyTh}). We observe that in the embedding ${B}^{r,\mathcal{L}}_{p,p}(G)\hookrightarrow {L}_r^{p,\mathcal{L}}(G)\hookrightarrow {B}^{r,\mathcal{L}}_{p,2}(G),$ if   $p=2$ then ${L}_r^{2,\mathcal{L}}(G)={H}^{r,\mathcal{L}}(G)={B}^{r,\mathcal{L}}_{2,2}(G).$ This shows $(4).$ Now, for the proof of $(5)$ we choose $f\in C^\infty(G).$ The Fourier inversion formula implies,
\begin{align*}
\Vert f \Vert_{L^{q}}&=\Vert \sum_{[\xi]\in \widehat{G}}\textrm{Tr}[\xi(x)\widehat{f}(\xi)]   \Vert_{L^q}\\
&=\Vert\sum_{s\in\mathbb{N}_0}\sum_{[\xi]\in \widehat{G}}\textrm{Tr}[\xi(x)\psi_{s}(\xi)\widehat{f}(\xi)]  \Vert_{L^q}\\
&\leq \sum_{s\in\mathbb{N}_0} \Vert\sum_{[\xi]\in \widehat{G}}\textrm{Tr}[\xi(x)\psi_{s}(\xi)\widehat{f}(\xi)]   \Vert_{L^q}\\
&= \sum_{s\in\mathbb{N}_0} \Vert \psi_{s}(D)f\Vert_{L^q}\leq  \sum_{s\in\mathbb{N}_0} 2^{{Q}(\frac{1}{p}-\frac{1}{q })}\Vert \psi_{s}(D)f\Vert_{L^p}\\
&=\Vert f\Vert_{{B}^{{Q}(\frac{1}{p}-\frac{1}{q}),\mathcal{L}}_{p,1}},
\end{align*}
where in the last line we have used Theorem \ref{Nikolskii}. This show that ${B}^{{Q}(\frac{1}{p}-\frac{1}{q}),\mathcal{L}}_{p,1}(G)\hookrightarrow L^{q}(G).$ The embedding  $L^{q}(G)\hookrightarrow {B}^{0,\mathcal{L}}_{q,\infty}(G)$ for $1<q\leq \infty,$ will be proved in Remark \ref{RemarkExtra1}.  This completes the proof.
\end{proof}

\section{Subelliptic Besov spaces vs Besov spaces associated to the Laplacian}\label{Sect5}
 In this section we study embedding properties between Besov spaces associated to the sub-Laplacian and the Besov spaces associated to the Laplacian. For our further analysis we require some properties of the global pseudo-differential operators as well as some Fourier multiplier theorems in the context of the matrix-valued quantization.

 \subsection{Pseudo-differential operators on compact Lie groups}
 A central tool for obtaining embedding properties between subelliptic Sobolev and  Besov spaces are pseudo-differential operators. According to the matrix-valued quantisation procedure developed in \cite{Ruz}, to every continuous linear operator $A$ on $G$ mapping $C^{\infty}(G)$ into $\mathscr{D}'(G)$ one can 
 associate a matrix-valued (full-symbol)
 \begin{equation}
     \sigma_A:G\times\widehat{G}\rightarrow \bigcup_{[\xi]\in \widehat{G}}\mathbb{C}^{d_\xi\times d_{\xi}},\,\,\,\sigma(x,\xi):=\sigma(x,[\xi])\in \mathbb{C}^{d_\xi\times d_\xi},
 \end{equation}
 satisfying 
 \begin{equation}\label{Pseudo}
 Af(x)=\sum_{[\xi]\in \widehat{G}}d_{\xi}\text{Tr}[\xi(x)\sigma_A(x,\xi)\widehat{f}(\xi)]. 
\end{equation}
In terms of the operator $A,$ we can recover the symbol $\sigma_A$ by the identity
\begin{equation*}
    \sigma_A(x,\xi)=\xi(x)^*(A\xi)(x),\,\,\,(A\xi)(x):=[A\xi_{ij}(x)]_{i,j=1}^{d_\xi}.
\end{equation*}

The matrix-valued calculus developed trough this quantisation process given by \eqref{Pseudo} is possible based on the definition of a differential structure on the unitary dual $\widehat{G}.$ Indeed, if we want to measure the decaying or increasing of matrix-valued symbols we need the notion of discrete derivatives (difference operators) $\Delta^\alpha_{\xi}$. According to \cite{RuzhanskyWirth2015},  we say that $Q_\xi$ is a difference operator of order $k,$ if it is defined by
\begin{equation}
    Q_\xi\widehat{f}(\xi)=\widehat{qf}(\xi),\,[\xi]\in \widehat{G}, 
\end{equation} for all $f\in C^\infty(G),$ for some function $q$ vanishing of order $k$ at the identity $e=e_G.$ The set $\textnormal{diff}^k(\widehat{G})$ denotes the set of all difference operators of order $k.$ For a such given smooth function $q,$ the associated difference operator will be denoted $\Delta_q:=Q_\xi.$ In order to define the H\"ormander classes in the matrix-valued context, we will choose an admissible selection of difference operators (see e.g. \cite{RuzhanskyDelgado2017,RuzhanskyWirth2015}),
\begin{equation*}
  \Delta_{\xi}^\alpha:=\Delta_{q_1}^{\alpha_1}\cdots   \Delta_{q_i}^{\alpha_{i}},\,\,\alpha=(\alpha_j)_{1\leq j\leq i}, 
\end{equation*}
where
\begin{equation}
    \textnormal{rank}\{\nabla q_j(e):1\leq j\leq i \}=\textnormal{dim}(G), \textnormal{   and   }\Delta_{q_i}\in \textnormal{diff}^{1}(\widehat{G}).
\end{equation}
\begin{remark}
Difference operators can be defined by using the representation theory on the group $G.$ Indeed, if $\xi_{0}$ is a fixed irreducible  representation, the matrix-valued  difference operator is the given by $\mathbb{D}_{\xi_0}=(\mathbb{D}_{\xi_0,i,j})_{i,j=1}^{d_{\xi_0}}=\xi_{0}(\cdot)-I_{d_{\xi_0}}$. If the representation is fixed we omit the index $\xi_0$ so that, from a sequence $\mathbb{D}_1=\mathbb{D}_{\xi_0,j_1,i_1},\cdots, \mathbb{D}_n=\mathbb{D}_{\xi_0,j_n,i_n}$ of operators of this type we define $\mathbb{D}^{\alpha}=\mathbb{D}_{1}^{\alpha_1}\cdots \mathbb{D}^{\alpha_n}_n$, where $\alpha\in\mathbb{N}^n$.
\end{remark}
\begin{remark}
By using the collection of fundamental representations, a set of difference operators $\{\Delta^\alpha:\alpha\in \mathbb{N}^n_0\}$ can be constructed (see Fischer \cite{FischerDiff}). We refer to Fischer \cite{Fischer2015} for an equivalent presentation to the calculus developed in \cite{Ruz} and \cite{RuzTurIMRN}.
\end{remark}

The following step is to provide the definition of  pseudo-differential operators on compact Lie groups associated to  H\"ormander classes (via localisations) and how this definition can be recovered from the notion of global symbol by using difference operators.  So, if $U$ is a open subset of $\mathbb{R}^n,$ we say that the function $a:U\times \mathbb{R}^n\rightarrow \mathbb{C},$ belongs to the H\"ormander class $S^m_{\rho,\delta}(U\times \mathbb{R}^n),$ $0\leq \rho,\delta\leq 1,$ if for every compact subset $K\subset U,$ we have   the symbol inequalities,
\begin{equation}
    |\partial_{x}^\beta\partial_{\xi}^\alpha a(x,\xi)|\leq C_{\alpha,\beta,K}(1+|\xi|)^{m-\rho|\alpha|+\delta|\beta|},
\end{equation} uniformly in $x\in K$ and $\xi\in \mathbb{R}^n.$ In this case, a  continuous linear operator $A$ from $C^\infty_0(U)$ into $C^\infty(U)$
is a pseudo-differential operator of order $m,$ in the $(\rho,\delta)$-class, if there exists
a function $a\in S^m_{\rho,\delta}(U\times \mathbb{R}^n),$ such that
\begin{equation*}
    Af(x)=\int\limits_{\mathbb{R}^n}e^{-2\pi i x\cdot \xi}a(x,\xi)(\mathscr{F}_{\mathbb{R}^n}{f})(\xi)d\xi,
\end{equation*} for all $f\in C^\infty_0(U),$ where
$
    (\mathscr{F}_{\mathbb{R}^n}{f})(\xi)=\int\limits_Ue^{-i2\pi x\cdot \xi}f(x)dx,
$ is the  Fourier transform of $f$ at $\xi.$ The class $S^m_{\rho,\delta}(U\times \mathbb{R}^n)$ is stable under coordinate changes if  $\rho\geq 1-\delta,$ although a symbolic calculus is only possible for $\delta<\rho$ and $\rho\geq 1-\delta.$    In the case of a $C^\infty$-manifold $M,$ (and consequently on every compact Lie group $G$) a linear continuous operator $A:C^\infty_0(M)\rightarrow C^\infty(M) $ is a pseudo-differential operator of order $m,$ in the $(\rho,\delta)$-class, $ \rho\geq 1-\delta, $ if for every local  coordinate patch $\kappa: M_{\kappa}\subset M\rightarrow U\subset \mathbb{R}^n,$
and for every $\phi,\psi\in C^\infty_0(U),$ the operator
\begin{equation}
    Tu:=\phi(\kappa^{-1})^*A\kappa^{*}(\phi u),\,\,u\in C^\infty(U),
\end{equation} is a pseudo-differential operator with symbol $a_T\in S^m_{\rho,\delta}(U\times \mathbb{R}^n).$ The operators $\kappa^{*}$ and $(\kappa^{-1})^*$ are the pullbacks, induced by the maps $\kappa$ and $\kappa^{-1}$ respectively. In this case we write $A\in \Psi^m_{\rho,\delta}(M,\textnormal{loc}).$ If $M=G$ is a compact Lie group, and $A\in \Psi^m_{\rho,\delta}(G,\textnormal{loc}),$ $\rho\geq 1-\delta,$ the matrix-valued symbol $\sigma_A$ of $A$ satisfies (see \cite{Fischer2015,Ruz,RuzhanskyTurunenWirth2014}),
\begin{equation}\label{HormanderSymbolMatrix}
    \Vert \partial_{x}^\beta \mathbb{D}^\gamma \sigma_A(x,\xi)\Vert_{op}\leq C_{\alpha,\beta}
    \langle \xi \rangle^{m-\rho|\gamma|+\delta|\beta|}
\end{equation} for all $\beta$ and  $\gamma $ multi-indices and all $(x,[\xi])\in G\times \widehat{G}$. Now, if $0\leq \delta,\rho\leq 1,$
we say that $\sigma_A\in \mathscr{S}^m_{\rho,\delta}(G),$ if the global symbol inequalities  \eqref{HormanderSymbolMatrix} hold true. So, for $\sigma_A\in \mathscr{S}^m_{\rho,\delta}(G)$ and $A$ defined by \eqref{Pseudo}, we write $A\in\textnormal{Op}(\mathscr{S}^m_{\rho,\delta}(G)).$  We have mentioned above that
\begin{equation}
   \textnormal{Op}(\mathscr{S}^m_{\rho,\delta}(G))= \Psi^m_{\rho,\delta}(G,\textnormal{loc}),\,\,\,0\leq \delta<\rho\leq 1,\,\rho\geq 1-\delta.
\end{equation} However, a symbolic calculus for the classes $\textnormal{Op}(\mathscr{S}^m_{\rho,\delta}(G)),$ $m\in\mathbb{R}$ and $0\leq \delta<\rho\leq 1$ (without the condition $\rho\geq 1-\delta$) has been established in \cite{Ruz}. Next, we record some of the mapping properties of these classes on Lebesgue spaces. We start with the following $L^p$-Fourier multiplier theorem which was presented as Corollary 5.1 in \cite{RuzhanskyWirth2015}.
\begin{theorem}\label{Mihlincondition}
Let $0\leq \rho\leq 1,$ and let $\tilde\varkappa$ be the smallest even integer larger than $n/2,$ $n:=\dim(G).$ If $A:C^\infty(G)\rightarrow\mathscr{D}'(G)$ is left-invariant and its matrix symbol $\sigma_A$ satisfies 
\begin{equation}
    \Vert \mathbb{D}^\gamma \sigma_A(\xi)\Vert_{op}\leq C_\gamma\langle \xi \rangle^{-\rho|\alpha|},\,\,[\xi]\in \widehat{G},\,|\gamma|\leq \tilde\varkappa,
\end{equation}then  $A$ extends to a bounded operator from $L_{\tilde\varkappa(1-\rho)\left|\frac{1}{2}-\frac{1}{p}\right|}^{p}(G)$ into $L^p(G)$ for all $1<p<\infty.$
\end{theorem} 
\begin{remark}\label{removeeveness}
Theorem \ref{Mihlincondition} was obtained in \cite{RuzhanskyWirth2015} under the condition of evenness for $\tilde\varkappa.$ However if we replace the collection of differences operators $\{\mathbb{D}^\alpha\}_{\alpha\in\mathbb{N}^n_0}$ by the collection of differences operators $\{\Delta^\alpha\}_{\alpha\in\mathbb{N}^n_0}$ associated to fundamental representations of the group $G$, $\tilde\varkappa$ can be replaced by $\varkappa:=[n/2]+1$, that is, the smallest  integer larger than $n/2,$ $n=\dim(G),$ (see Remark 4.10 of \cite{RuzhanskyDelgado2017}).
\end{remark}

For non-invariant operators, Theorem \ref{Mihlincondition} and Remark \ref{removeeveness} allow us to obtain the following version of Theorem 1.1 presented in \cite{RuzhanskyDelgado2017}.

\begin{theorem}\label{DelgadoRuzhanskylimitedregularity}
Let $0\leq \delta, \rho\leq 1,$ and let $\varkappa$ be the smallest  integer larger than $n/2,$ $n:=\dim(G).$ If $A:C^\infty(G)\rightarrow\mathscr{D}'(G)$ is a continuous operator such that its matrix symbol $\sigma_A$ satisfies 
\begin{equation}
    \Vert \partial_{x}^\beta\Delta^\gamma \sigma_A(x,\xi)\Vert_{op}\leq C_{\gamma,\beta}\langle \xi \rangle^{-m_0-\rho|\alpha|+\delta|\gamma|},\,\,[\xi]\in \widehat{G},\,|\gamma|\leq \varkappa, |\beta|\leq [n/p]+1,
\end{equation}
with
\begin{equation}
    m_0\geq \varkappa(1-\rho)|\frac{1}{p}-\frac{1}{2}|+\delta([n/p]+1).
\end{equation}
Then  $A$ extends to a bounded operator from $L^{p}(G)$ into $L^p(G)$ for all $1<p<\infty.$
\end{theorem} For operators with smooth symbols, in  \cite{RuzhanskyDelgado2017} is proved the following $L^p$-pseudo-differential theorem.
\begin{theorem}\label{DelgadoRuzhanskyLppseudo}
Let $G$ be a compact Lie group of dimension $n.$ Let $0\leq \delta<\rho\leq 1$ and let
\begin{equation}
    0\leq \nu<\frac{n(1-\rho)}{2},
\end{equation} for $0< \rho<1$ and $\nu=0$ for $\rho=1.$ Let $\sigma\in \mathscr{S}^{-\nu}_{\rho,\delta}(G).$ Then $A\equiv \sigma(x,D)$ extends to a bounded operator on $L^p(G)$ provided that
\begin{equation}
    \nu\geq n(1-\rho)|\frac{1}{p}-\frac{1}{2}|.
\end{equation}
\end{theorem}
The sharpness of the order $\frac{n(1-\rho)}{2}$ in Theorem \ref{DelgadoRuzhanskyLppseudo} was discussed in Remark 4.16 of \cite{RuzhanskyDelgado2017}.

 A fundamental tool in our further analysis is the following estimate for operators associated to compactly supported symbols in $\xi$.
 \begin{lemma}\label{DelgadoRuzhanskyModifiedLemma}
 Let $G$ be a compact Lie group of dimension $n$. Let $\{A_j\}_{j\in \mathbb{N}}$ be a sequence of Fourier multipliers with full symbols $\sigma_{A_j}\in \mathscr{S}^{0}_{1,0}({G}),$  supported in 
 \begin{equation}
    \Sigma_j:= \{[\xi]\in \widehat{G}: \frac{1}{2}j\leq \langle \xi \rangle<j\}.
 \end{equation} If $\sup_{j}\Vert {A_j}\Vert_{\mathscr{B}(L^2(G))}<\infty,$ then every $A_j$ extends to a bounded operator on $L^p(G),$ $1<p\leq \infty,$ and   $\sup_{j}\Vert {A_j}\Vert_{\mathscr{B}(L^p(G))}<\infty.$ 
 \end{lemma}
 \begin{proof}
 The case $p=\infty$ is only a special case of  Lemma 4.11 of \cite{RuzhanskyDelgado2017} for $a=0$. Now, we can use the Riesz-Thorin interpolation theorem  between the estimates $$\sup_{j}\Vert {A_j}\Vert_{\mathscr{B}(L^{p_i})}<\infty,\,\,i=0,1,$$ for $p_0=2$ and $p_1=\infty$ in order to obtain  Lemma \ref{DelgadoRuzhanskyModifiedLemma} for $2\leq p<\infty.$ The case $1<p\leq 2 $ now follows by an argument of duality. The proof is complete.
 \end{proof}

 The following Lemma will be useful for our further analysis (see Lemma 4.5 of \cite{RuzhanskyTurunenWirth2014}).
 \begin{lemma}\label{LemmaforHypoelliptic}
 Let $m\geq m_0,$ $0\leq \delta< \rho\leq 1,$ and let us fix difference operators $\{\mathbb{D}_{ij}\}_{1\leq i,j\leq d_{\xi_0}}$ corresponding to some representation $\xi_0\in [\xi_{0}]\in \widehat{G}.$ Let the matrix symbol $a=a(x,\xi)$ satisfy
 \begin{equation}
 \Vert \mathbb{D}^\gamma\partial_x^\beta a(x,\xi)\Vert_{op}\leq C_{\gamma\beta}\langle \xi \rangle^{m-\rho|\gamma|+\delta|\beta|}
 \end{equation}for all multi-index $\beta$ and $\gamma.$ Assume that $a(x,\xi)$ is invertible for all $x\in G$ and $[\xi]\in \widehat{G}$ and satisfies 
 \begin{equation}
 \Vert a(x,\xi)^{-1}\Vert_{op}\leq C\langle \xi \rangle^{-m_0},\,\,x\in G,\,[\xi]\in \widehat{G},
 \end{equation} and if $m_0\neq m,$ in addition that
 \begin{equation}
     \Vert a(x,\xi)^{-1}[\mathbb{D}^\gamma\partial_x^\beta a(x,\xi)]\Vert_{op}\leq C_{\gamma\beta}\langle \xi \rangle^{\rho|\gamma|+\delta|\beta|}.
 \end{equation}Then the following inequalities
 \begin{equation}
 \Vert \mathbb{D}^\gamma\partial_x^\beta[ a(x,\xi)^{-1}]\Vert_{op}\leq C_{\gamma\beta}'\langle \xi \rangle^{m_0-\rho|\gamma|+\delta|\beta|}
 \end{equation} hold true for all multi-indices $\beta$ and $\gamma.$
 \end{lemma}

\subsection{Embedding properties}

Now, we compare subelliptic Besov spaces associated to the sub-Laplacian with Besov spaces associated to the Laplacian on the group. We first analyse it for Sobolev spaces modelled on $L^p$-spaces. Our starting point is the following estimate (see Proposition 3.1 of \cite{GarettoRuzhansky2015}), 
\begin{equation}
  c\langle \xi\rangle^{\frac{1}{\kappa}}\leq  1+\nu_{ii}(\xi)\leq \sqrt{2}\langle \xi\rangle,\,\,\langle \xi\rangle:=(1+\lambda_{[\xi]})^\frac{1}{2},
 \end{equation} where $\kappa$ is the order in the H\"ormander condition. We immediately have the following estimate,
 \begin{equation}\label{GarettoRuzhanskyIneq}
  \langle \xi\rangle^{\frac{1}{\kappa}}\lesssim  (1+\nu_{ii}(\xi)^2)^{\frac{1}{2}}\lesssim \langle \xi\rangle.
\end{equation}
\begin{remark}[Order for negative powers of $(1+\mathcal{L})^{\frac{1}{2}}$]\label{orderforpowers}   
For $1\leq k<n:=\dim G$ and  $s>0,$ define the operator $\mathcal{M}_s:=(1+\mathcal{L})^{\frac{s}{2}}.$ We will show that although $\mathcal{M}_{s}\in \textnormal{Op}(\mathscr{S}^{s}_{1,0}(G)),$ we have $\mathcal{M}_{-s}:=(1+\mathcal{L})^{-\frac{s}{2}}\in \textnormal{Op}(\mathscr{S}^{-s/\kappa}_{-1/\kappa,0}(G)).$ The matrix-valued symbol associated to $\mathcal{M}=\mathcal{M}_1$ which we denote by $\widehat{ \mathcal{M}}_1(\xi)$ is given by \begin{equation}\widehat{ \mathcal{M} }_1(\xi)=\textnormal{diag}[(1+\nu_{ii}(\xi)^2)^{\frac{1}{2}}]_{1\leq i\leq d_\xi},\,\,[\xi]\in \widehat{G},\end{equation}
\end{remark}
\noindent by a suitable choice of the basis in the representation spaces. Because, $\mathcal{M}_2\in \textnormal{Op}(\mathscr{S}^{2}_{1,0}(G)),$ the matrix valued functional calculus allow us to conclude that $\mathcal{M}_1\in \textnormal{Op}(\mathscr{S}^{1}_{1,0}(G)),$ (see Corollary 4.3 of \cite{RuzhanskyWirth2014}). Therefore, we have that the matrix valued-symbol $\widehat{ \mathcal{M} }_1$ satisfies estimates of the type,
\begin{equation}
 \Vert \mathbb{D}^\gamma \widehat{\mathcal{M}  }_1(\xi)\Vert_{op}\leq C_{\gamma}\langle \xi \rangle^{1-|\gamma|}\leq C_\gamma\langle \xi \rangle^{1-\rho|\gamma|},\,0<\rho\leq 1,
 \end{equation}for all multi-indices $\beta$ and $\gamma.$ On the other hand, from \eqref{GarettoRuzhanskyIneq}, we have
 \begin{equation}
     \Vert \widehat{\mathcal{M}  }_1(\xi)^{-1}\Vert_{op}=\sup_{1\leq i\leq d_\xi}(1+\nu_{ii}^{2}(\xi))^{-\frac{1}{2}}\lesssim \langle \xi\rangle^{-\frac{1}{\kappa}}.
 \end{equation}
If $|\gamma|=1,$ observe  that 
\begin{equation}
     \Vert \widehat{ \mathcal{M} }_1(\xi)^{-1}[\mathbb{D}^\gamma \widehat{ \mathcal{M} }_1(\xi)]\Vert_{op}\leq \Vert \widehat{ \mathcal{M}  }_1(\xi)^{-1}\Vert_{op}\Vert \mathbb{D}^\gamma \widehat{\mathcal{M}  }_1(\xi)\Vert_{op}\lesssim   C_{\gamma\beta}\langle \xi \rangle^{-1/\kappa}.
 \end{equation} If we iterated this process for $|\gamma|\geq 2,$  we can prove that
\begin{equation}
     \Vert \widehat{ \mathcal{M} }_1(\xi)^{-1}[\mathbb{D}^\gamma \widehat{ \mathcal{M} }_1(\xi)]\Vert_{op}\leq \Vert \widehat{ \mathcal{M} }_1(\xi)^{-1}\Vert_{op}\Vert \mathbb{D}^\gamma \widehat{ \mathcal{M}}_1(\xi)\Vert_{op}\lesssim   C_{\gamma\beta}\langle \xi \rangle^{-|\gamma|/\kappa}.
\end{equation} Indeed,
\begin{align*}
    \Vert \widehat{ \mathcal{M} }_1(\xi)^{-1}\Vert_{op}\Vert \mathbb{D}^\gamma \widehat{ \mathcal{M}}_1(\xi)\Vert_{op}\lesssim \langle \xi\rangle^{-1/\kappa} \langle \xi\rangle^{-(|\gamma|-1)}\leq \langle \xi\rangle^{ {-1/\kappa} -(|\gamma|-1)/\kappa}= \langle \xi\rangle^{-|\gamma|/\kappa}.
\end{align*}
According to Lemma \ref{LemmaforHypoelliptic} and Theorem 4.2 of \cite{RuzhanskyWirth2014}, the analysis above lead us to conclude that $\mathcal{M}_{-s}=(1+\mathcal{L})^{-\frac{s}{2}}\in \textnormal{Op}(\mathscr{S}^{-s/\kappa}_{\,\,1/\kappa,0}(G))$ for all $s>0.$

In order to compare  subelliptic Besov spaces with Besov spaces associated to the Laplacian, we  start by considering the problem for Sobolev spaces and later extend it to Besov spaces. So, we have the following theorem. Here, $\varkappa:=[n/2]+1$, is the smallest  integer larger than $n/2,$ $n=\dim(G).$ 
\begin{theorem}\label{SubHvsEllipH} Let $G$ be a compact Lie group and let us consider the sub-Laplacian $\mathcal{L}=-(X_1^2+\cdots +X_k^2),$ where the system of vector fields $X=\{X_i\}_{i=1}^k$ satisfies the H\"ormander condition of order $\kappa$.
Let $s\geq 0,$ $\varkappa:=[\dim(G)/2]+1,$ and  $1<p<\infty.$ Then we have the continuous embeddings
\begin{equation}
L_s^{p}(G)\hookrightarrow L_s^{p,\mathcal{L}}(G) \hookrightarrow L_{ \frac{s}{\kappa}-\varkappa(1-\frac{1}{\kappa})\left|\frac{1}{2}-\frac{1}{p}\right|  }^{p}(G).
\end{equation} More precisely, for every $s\geq 0$ there exist constants $C_a>0$ and $C_{b}>0$ satisfying,
\begin{equation}\label{Embd1}
   C_a \Vert f\Vert_{ L_{  s_{\kappa,\varkappa} }^{p}(G) }\leq \Vert f\Vert_{L_s^{p,\mathcal{L}}(G)},\,\,f\in L_s^{p,\mathcal{L}}(G),
\end{equation}
where $s_{\kappa,\varkappa}:=\frac{s}{\kappa}-\varkappa(1-\frac{1}{\kappa})\left|\frac{1}{2}-\frac{1}{p}\right|,$ and 
\begin{equation}\label{Embd2}
    \Vert f\Vert_{ L_s^{p,\mathcal{L}}(G)}\leq C_b\Vert f\Vert_{L_s^{p}(G)},\,\,f\in L_s^{p}(G).
\end{equation}
Consequently, we have the following embeddings
\begin{equation}
L_{ -\frac{s}{\kappa}+\varkappa(1-\frac{1}{\kappa})\left|\frac{1}{2}-\frac{1}{p}\right| }^{p}(G)\hookrightarrow L_{-s}^{p,\mathcal{L}}(G) \hookrightarrow L_{-s}^{p}(G) .
\end{equation}
\end{theorem}
\begin{proof}
The proof will be based on the continuity properties of the operators
\begin{equation}
    \mathcal{R}_s:=(1+\mathcal{L}_G)^{\frac{s}{2}} (1+\mathcal{L})^{-\frac{s}{2}} \textnormal{  and  }  \mathcal{R}_s^{'}:=(1+\mathcal{L}_G)^{-\frac{s}{2}} (1+\mathcal{L})^{\frac{s}{2}},\,\,s\geq 0.
\end{equation} From Remark \ref{orderforpowers} we deduce that $\mathcal{R}_s\in \textnormal{Op}(\mathscr{S}^{s-\frac{s}{\kappa}}_{1/\kappa,0}(G))$   and $\mathcal{R}_s'\in  \textnormal{Op}(\mathscr{S}^{0}_{1,0}(G)).$
 Let us observe that $\mathcal{R}_s'\in \textnormal{Op}(\mathscr{S}^{0}_{1,0}(G))$ implies that for all $s\geq 0,$ the operator $\mathcal{R}_s'$ extends to a bounded operator on $L^p(G)$ for all $1<p<\infty.$ So, for $C_b:=\Vert \mathcal{R}_s' \Vert_{\mathscr{B}(L^p(G))},$ we have
\begin{equation}
    \Vert (1+\mathcal{L}_G)^{-\frac{s}{2}} (1+\mathcal{L})^{\frac{s}{2}} f\Vert_{L^p(G)}\leq C_b\Vert f\Vert_{L^p(G)},
\end{equation} which in turns implies 
\begin{equation}
    \Vert (1+\mathcal{L}_G)^{-\frac{s}{2}}  f\Vert_{L^p(G)}\leq C_b\Vert (1+\mathcal{L})^{-\frac{s}{2}}f\Vert_{L^p(G)}.
\end{equation} So, we provide the embedding $L_{-s}^{p,\mathcal{L}}(G) \hookrightarrow L_{-s}^{p}(G)$ and by duality we have $ L_s^{p}(G) \hookrightarrow L_s^{p,\mathcal{L}}(G).$ The last embedding is equivalent to the inequality
    \begin{equation}
    \Vert f\Vert_{ L_s^{p,\mathcal{L}}(G) }\leq C_b \Vert f\Vert_{L_s^{p}(G)},\,\,f\in L_s^{p}(G).
\end{equation} For the second part of the proof, let us define  $$T_s:=(1+\mathcal{L}_G)^{\frac{1}{2}(\frac{s}{\kappa}-s)}\mathcal{R}_{s}=(1+\mathcal{L}_G)^{\frac{s}{2\kappa}}(1+\mathcal{L})^{-\frac{s}{2}}.$$
Now, by Remark \ref{orderforpowers} the operator $T_s$ belongs to the class $\textnormal{Op}(\mathscr{S}^{0}_{\rho,0}(G)),$ $\rho=1/\kappa,$ and by Theorem \ref{Mihlincondition} and Remark \ref{removeeveness} we obtain its boundedness from $L_{ \varkappa(1-\frac{1}{\kappa})\left|\frac{1}{2}-\frac{1}{p} \right|  }^{p}(G)$ into $L^p(G)$ for all $1<p<\infty,$ where $\varkappa:=[n/2]+1.$ Consequently, we obtain 
\begin{equation}\label{Auxiliar23}
C_a\Vert T_s f\Vert_{   L^p(G)}\leq \Vert f \Vert_{ L_{ \varkappa(1-\frac{1}{\kappa})\left|\frac{1}{2}-\frac{1}{p} \right| }^{p}(G)  },\,\,C_{a}:=1/\Vert T\Vert_{\mathscr{B}( L_{ \varkappa(1-\frac{1}{\kappa})\left|\frac{1}{2}-\frac{1}{p} \right| }^{p},L^p)}.
\end{equation} Similar to the first part of the proof, we have
\begin{equation}\label{auxiliarprime}
    C_a\Vert (1+\mathcal{L})^{-\frac{s}{2}}f \Vert_{L^p(G)}\leq\Vert f\Vert_{L_{ \varkappa(1-\frac{1}{\kappa})\left|\frac{1}{2}-\frac{1}{p} \right|-\frac{s}{\kappa} }^{p}(G)}.
\end{equation}
Indeed, by using that the operators $(1+\mathcal{L}_G)^{\frac{s}{2\kappa}}$ and $(1+\mathcal{L})^{-\frac{s}{2}}$ commute,  from \eqref{Auxiliar23}, we have the inequality,
\begin{equation}\label{auxiliar32}
C_a\Vert (1+\mathcal{L})^{-\frac{s}{2}} (1+\mathcal{L}_G)^{\frac{s}{2\kappa}}f\Vert_{   L^p(G)}\leq \Vert f \Vert_{ L_{ \varkappa(1-\frac{1}{\kappa})\left|\frac{1}{2}-\frac{1}{p} \right| }^{p}(G)  },
\end{equation}
which implies \eqref{auxiliarprime} if in \eqref{auxiliar32} we change $f$ by $(1+\mathcal{L}_G)^{-\frac{s}{2\kappa}}f.$ So, 
we conclude that $L_{-\frac{s}{\kappa}+\varkappa(1-\frac{1}{\kappa})\left|\frac{1}{2}-\frac{1}{p}\right|  }^{p}(G)\hookrightarrow L_{-s}^{p,\mathcal{L}}(G),$ and by duality we obtain the embedding $L_s^{p,\mathcal{L} }(G)\hookrightarrow L^{p }_{\frac{s}{\kappa}-\varkappa(1-\frac{1}{\kappa})\left|\frac{1}{2}-\frac{1}{p}\right|  }(G),$ which implies the estimate,
\begin{equation}
   C_a \Vert f\Vert_{ L_{ s_{\kappa,\varkappa} }^{p}(G) }\leq \Vert f\Vert_{L_s^{p,\mathcal{L}}(G)},\,\,f\in L_s^{p,\mathcal{L}}(G),
\end{equation}
when $s_{\kappa,\varkappa}=\frac{s}{\kappa}-\varkappa(1-\frac{1}{\kappa})\left|\frac{1}{2}-\frac{1}{p}\right|.$ Thus, we finish the proof.
\end{proof}
\begin{remark}
For $p=2,$ Theorem \ref{SubHvsEllipH} was obtained in \cite{GarettoRuzhansky2015}. In this case for $s>0,$ we have $s_{\varkappa,\kappa}=s/\kappa.$ Consequently, 
\begin{equation*}
    H^{s}(G)\hookrightarrow H^{s,\mathcal{L}}(G) \hookrightarrow H^{\frac{s}{\kappa}}(G) \hookrightarrow L^2(G)  \hookrightarrow  H^{-\frac{s}{\kappa}}(G)\hookrightarrow H^{-s,\mathcal{L}}(G) \hookrightarrow H^{-s}(G).
\end{equation*}
Also, as it was pointed out in \cite{GarettoRuzhansky2015}, for integers $s\in\mathbb{N}$ the embedding $H^{ s,\mathcal{L} }(G)\hookrightarrow H^{ \frac{s}{\kappa} }(G) ,$ is in fact, Theorem 13 in \cite{RothschildStein76}.
\end{remark}
Now, we prove the main theorem of this subsection.

\begin{theorem}\label{SubBesovvsEllipBesov} Let $G$ be a compact Lie group and let us consider the sub-Laplacian $\mathcal{L}=-(X_1^2+\cdots +X_k^2),$ where the system of vector fields $X=\{X_i\}_{i=1}^k$ satisfies the H\"ormander condition of order $\kappa$.
Let $s\geq 0,$ $0<q\leq \infty,$ and $1<p< \infty.$ Then we have the continuous embeddings
\begin{equation}
B^{s}_{p,q}(G)\hookrightarrow B^{s,\mathcal{L}}_{p,q}(G) \hookrightarrow B^{\frac{s}{\kappa}-\varkappa(1-\frac{1}{\kappa})\left|\frac{1}{2}-\frac{1}{p}\right|}_{p,q}(G).
\end{equation} More precisely, for every $s\geq 0$ there exist constants $C_a>0$ and $C_{b}>0$ satisfying,
\begin{equation}
   C_a \Vert f\Vert_{ B^{s_{\kappa,\varkappa}}_{p,q}(G) }\leq \Vert f\Vert_{B^{s,\mathcal{L}}_{p,q}(G)},\,\,f\in B^{s,\mathcal{L}}_{p,q}(G),
\end{equation}
where $s_{\kappa,\varkappa}:=\frac{s}{\kappa}-\varkappa(1-\frac{1}{\kappa})\left|\frac{1}{2}-\frac{1}{p}\right|,$ and 
\begin{equation}
    \Vert f\Vert_{B^{s,\mathcal{L}}_{p,q}(G)}\leq C_b\Vert f\Vert_{B^{s}_{p,q}(G)},\,\,f\in B^{s}_{p,q}(G).
\end{equation}
Consequently, we have the following embeddings
\begin{equation}
B^{-\frac{s}{\kappa}+\varkappa(1-\frac{1}{\kappa})\left|\frac{1}{2}-\frac{1}{p}\right|}_{p,q}(G)\hookrightarrow B^{-s,\mathcal{L}}_{p,q}(G) \hookrightarrow B^{-s}_{p,q}(G).
\end{equation}
\end{theorem}
\begin{proof}
First, let us observe that for every $s\in \mathbb{R},$
\begin{equation}\label{estimatebesovsobolev}
    \Vert f\Vert_{B^{s,\mathcal{L}}_{p,q}(G)}\asymp \left\Vert   \{\Vert  \psi_j(D)f\Vert_{L_s^{p,\mathcal{L}}(G)} \}_{j\in\mathbb{N}}\right\Vert_{L^q(\mathbb{N})}.
\end{equation}
This can be proved by the functional calculus. Indeed, if $E_{\lambda}^{\mathcal{M}},$ $\lambda>0,$ denotes the spectral resolution associated to $\mathcal{M}=(1+\mathcal{L})^\frac{1}{2}$ and $0<q<\infty,$ we get
\begin{align*}
  2^{j s q} \Vert \psi_{j}(D)f\Vert_{L^p(G)}^q &= \left\Vert  \int\limits_{2^j}^{2^{j+1}} 2^{js } dE^{\mathcal{M}}_{\lambda}f \right\Vert^q_{L^p(G)}=\left\Vert \int\limits_{2^j}^{2^{j+1}} 2^{j s }\lambda^{-s}\lambda^{s} dE^{\mathcal{M}}_{\lambda}f \right\Vert^q_{L^p(G)}\\
  &\asymp \left\Vert \int\limits_{2^j}^{2^{j+1}} \lambda^{s} dE^{\mathcal{M}}_{\lambda}f \right\Vert^q_{L^p(G)}= \left\Vert \int\limits_{0}^{\infty} \lambda^{s} dE^{\mathcal{M}}_{\lambda} \int\limits_{2^j}^{2^{j+1}}  dE^{\mathcal{M}}_{\lambda}f \right\Vert^q_{L^p(G)}\\
  &=:\Vert (1+\mathcal{L})^\frac{s}{2}\psi_j(D)f \Vert_{L^p(G)}^q,
\end{align*} which proves the desired estimate \eqref{estimatebesovsobolev}. In the previous lines we have used the following property
\begin{equation*}
    \int\limits_{[a,b]\cap [c,d]}f(\lambda)g(\lambda)dE^{\mathcal{M}}_\lambda=\int\limits_{[a,b]}f(\mu)dE^{\mathcal{M}}_{\mu} \int\limits_{[c,d]}g(\lambda)dE^{\mathcal{M}}_\lambda,
\end{equation*}
of the functional calculus (see e.g. \cite{Weid}). For the second part of the proof, we choose a $S^1_0$-partition of unity (or Littlewood-Paley partition of the unity) $\{\tilde{\psi}_{j}\}.$ So, we assume the following properties (see e.g.  Definition 2.1 of Garetto \cite{Garetto}),
\begin{itemize}
    \item  $\tilde\psi_0 \in \mathscr{D}(-\infty,\infty),$ $\tilde\psi_0(t)=1$ for $|t|\leq 1,$ and  $\tilde\psi_0(t)=0$ for  $|t|\geq 2.$
    \item For every $j\geq 1, $ $\tilde\psi_{j}(t)=\tilde\psi_0(2^{-j}t)-\tilde\psi_{0}(2^{-j+1}t):=\tilde\psi_1(2^{-j+1}t),$
    \item $|\tilde{\psi}^{(\alpha)}_j(t)|\leq C_{\alpha}2^{-\alpha j},\,\,\alpha\in\mathbb{N}_0,$ and
\item $ \sum_{j} \tilde{\psi}_{j}(t)=1, $ $t\in\mathbb{R}.$
\end{itemize}

We will use the existence of an integer $j_0\in\mathbb{N}$ satisfying $\tilde{\psi}_{j}\psi_{j'}=0$ for all $j,j'$ with $|j-j'|\geq j_0,$  as follows. If $f\in C^\infty(G)$ and $\tilde\ell \in\mathbb{N},$ then
\begin{align*}
    \psi_{\tilde\ell}(D)f&=
    \sum_{j,j'}\tilde{\psi}_{j}((1+\mathcal{L})^\frac{1}{2})\psi_{\tilde\ell}(D)\tilde{\psi}_{j'}((1+\mathcal{L}_G)^\frac{1}{2})f\\
    &=\sum_{j,j'=\tilde\ell-j_0}^{\tilde\ell+j_0}\tilde{\psi}_{j}((1+\mathcal{L})^\frac{1}{2})\psi_{\tilde\ell}(D)\tilde{\psi}_{j'}((1+\mathcal{L}_G)^\frac{1}{2})f.
\end{align*}
So, the following estimate
\begin{equation}
    \Vert \psi_{\tilde\ell}(D)f  \Vert_{L^p}\leq \sum_{j,j'=\tilde\ell-j_0}^{\tilde\ell+j_0}\Vert \tilde{\psi}_{j}((1+\mathcal{L})^\frac{1}{2})\psi_{\tilde\ell}(D)\tilde{\psi}_{j'}((1+\mathcal{L}_G)^\frac{1}{2})f \Vert_{L^p},
\end{equation} holds true for every $\tilde\ell\in \mathbb{N}.$ From Lemma \ref{DelgadoRuzhanskyModifiedLemma} applied to every operator  $  \psi_{\tilde\ell}(D)\tilde{\psi}_{j}((1+\mathcal{L})^\frac{1}{2}) ,$ we have 
\begin{equation}
    \Vert \tilde{\psi}_{j}((1+\mathcal{L})^\frac{1}{2})\psi_{\tilde\ell}(D)\tilde{\psi}_{j'}((1+\mathcal{L}_G)^\frac{1}{2})f \Vert_{L^p}\leq C\Vert \tilde{\psi}_{j'}((1+\mathcal{L}_G)^\frac{1}{2})f \Vert_{L^p},
\end{equation} and
\begin{equation}
    \Vert \tilde{\psi}_{j}((1+\mathcal{L})^\frac{1}{2})\tilde{\psi}_{j'}((1+\mathcal{L}_G)^\frac{1}{2})f \Vert_{L^p}\leq C\Vert \tilde{\psi}_{j}((1+\mathcal{L})^\frac{1}{2})f \Vert_{L^p},
\end{equation}
where the constant $C$ is independent of $j,j'\in\mathbb{N}.$  Now, if $s\geq 0,$ we have 
\begin{align*}
    \Vert f\Vert_{B^{s,\mathcal{L}}_{p,q}} &=\left\Vert \left\{ 2^{j s}\Vert {\psi}_{j}(D)f \Vert_{L^p}  \right\}_{j\in\mathbb{N}}  \right\Vert_{\ell^q(\mathbb{N})}\\
    &\lesssim \left\Vert \left\{ \sum_{j'=\tilde\ell-j_0}^{\tilde\ell+j_0}2^{\tilde\ell s}\Vert \tilde{\psi}_{j'}((1+\mathcal{L}_G)^\frac{1}{2})f \Vert_{L^p}  \right\}_{\tilde\ell\in\mathbb{N}}  \right\Vert_{ \ell^q(\mathbb{N}) }\lesssim (2j_0+1) \Vert f\Vert_{B^{s}_{p,q}}. 
\end{align*}
This estimate proves the embedding $  B^{s}_{p,q}(G)\hookrightarrow B^{s,\mathcal{L}}_{p,q}(G)$ which by duality implies $  B^{-s,\mathcal{L}}_{p,q}(G)\hookrightarrow B^{-s}_{p,q}(G)$ for $s\geq 0.$ Now, for $s\geq 0$ and $s_{\kappa,\varkappa}:=\frac{s}{\kappa}-\varkappa(1-\frac{1}{\kappa})\left|\frac{1}{2}-\frac{1}{p}\right|,$
 we have
\begin{align*}
\Vert f\Vert_{ B^{s_{\kappa,\varkappa}}_{p,q}(G) } &\asymp \left\Vert   \{\Vert  \psi_{\tilde\ell }((1+\mathcal{L}_G)^{\frac{1}{2}})f\Vert_{L_{ s_{\kappa,\varkappa} }^{p}(G)} \}_{\tilde\ell\in\mathbb{N}}\right\Vert_{\ell^q(\mathbb{N})},
\end{align*}
which in turns implies,
\begin{align*}
\Vert f\Vert_{ B^{s_{\kappa,\varkappa}}_{p,q}(G) } &\leq  \left\Vert   \{   \sum_{j',j=\tilde\ell-j_0}^{\tilde\ell+j_0}\Vert  \psi_j((1+\mathcal{L}_G)^{\frac{1}{2}}) \psi_{j'}((1+\mathcal{L})^{\frac{1}{2}}) f\Vert_{L_{ s_{\kappa,\varkappa} }^{p}(G)} \}_{\tilde\ell\in\mathbb{N}}\right\Vert_{\ell^q(\mathbb{N})}\\
&\lesssim  \left\Vert   \{   \sum_{j',j=\tilde\ell-j_0}^{\tilde\ell+j_0}\Vert  \psi_{j'}((1+\mathcal{L})^{\frac{1}{2}}) f\Vert_{L_{s_{\kappa,\varkappa}}^{p}(G)} \}_{\tilde\ell\in\mathbb{N}}\right\Vert_{\ell^q(\mathbb{N})}\\
&\lesssim  \left\Vert   \{   \sum_{j',j=\tilde\ell-j_0}^{\tilde\ell+j_0}\Vert  \psi_{j'}((1+\mathcal{L})^{\frac{1}{2}}) f\Vert_{L_s^{p,\mathcal{L}}(G)} \}_{\tilde\ell\in\mathbb{N}}\right\Vert_{\ell^q(\mathbb{N})}\\
&\lesssim (2j_0+1)^2\left\Vert   \{   \Vert  \psi_{j'}((1+\mathcal{L})^{\frac{1}{2}}) f\Vert_{L_s^{p,\mathcal{L}}(G)} \}_{j'\in\mathbb{N}}\right\Vert_{\ell^q(\mathbb{N})}
\asymp\Vert f\Vert_{ B^{s,\mathcal{L}}_{p,q}(G) },
\end{align*}
where we have used \eqref{Embd1}. The last estimate proves the embedding $B^{s,\mathcal{L}}_{p,q}(G) \hookrightarrow B^{\frac{s}{\kappa}-\varkappa(1-\frac{1}{\kappa})\left|\frac{1}{2}-\frac{1}{p}\right|}_{p,q}(G).$ The duality argument implies that $B^{-s_{\kappa,\varkappa}}_{p,q}(G) \hookrightarrow B^{-s}_{p,q}(G)$ for every $s\geq 0.$ So, we finish the proof.
\end{proof}

\begin{remark}\label{RemarkExtra1}
With the notation of the previous result, the embedding  $L^{q}(G)\hookrightarrow {B}^{0,\mathcal{L}}_{q,\infty}(G),$ $1<q\leq \infty,$ is a consequence of the following estimate,
\begin{align*}
    \Vert f\Vert_{{B}^{0,\mathcal{L}}_{q,\infty} }= \sup_{\tilde\ell\in\mathbb{N}}\Vert \psi_{\tilde\ell}(D)f \Vert_{L^q}\lesssim \sup_{\tilde\ell\geq j_0}  \sum_{j=\tilde\ell-j_0}^{\tilde\ell+j_0}\Vert \tilde{\psi}_j((1+\mathcal{L})^\frac{1}{2})f \Vert_{L^q} \leq C(2j_0+1)\Vert f\Vert_{L^q}.
\end{align*}

\end{remark}
 \section{Subelliptic Triebel-Lizorkin spaces and interpolation of subelliptic Besov spaces}\label{Triebel}
 
 The Littlewood-Paley theorem \eqref{LitPaleyTh},
\begin{equation}
  \Vert f \Vert_{F^{0,\mathcal{L}}_{p,2}(G)    }:= \Vert   [\sum_{s\in\mathbb{N}_0}  |\psi_{s}(D)f(x)       |^{2}]^{\frac{1}{2}}\Vert_{L^{p}(G)}
\asymp \Vert f \Vert_{L^p(G)}, \end{equation} suggests the subelliptic functional class defined by the norm
\begin{equation}
  \Vert f \Vert_{F^{r,\mathcal{L}}_{p,q}(G)    }:= \Vert   [\sum_{s\in\mathbb{N}_0}  2^{srq}|\psi_{s}(D)f(x)       |^{q}]^{\frac{1}{q}}\Vert_{L^{p}(G)},
  \end{equation}
 for $0<p\leq \infty,$ and $0<q<\infty$ with the following modification
 \begin{equation}
  \Vert f \Vert_{F^{r,\mathcal{L}}_{p,\infty}(G)    }:= \Vert   \sup_{s\in\mathbb{N}_0}[  2^{sr}|\psi_{s}(D)f(x)       |]\Vert_{L^{p}(G)},
  \end{equation} for $q=\infty.$ By the analogy with ones presented in Triebel \cite{Triebel1983,Triebel2006} we will  refer to the spaces $F^{r,\mathcal{L}}_{p,q}(G),$ as subelliptic Triebel-Lizorkin spaces. The main objective of this section is to present the interpolation properties for these spaces and their relation with the subelliptic Besov spaces.
 In the context of compact Lie groups, Triebel-Lizorkin spaces $F^r_{p,q}(G)$ associated to the Laplacian $\mathcal{L}_G$ have been introduced in \cite{NurRuzTikhBesov2015,NurRuzTikhBesov2017}, by using a global description of the Fourier transform through the representation theory of the group $G$. Taking these references as a point of departure, we have the following theorem.
 \begin{theorem}
 Let $G$ be a compact Lie group of dimension $n.$ Then we have the following properties.
\begin{itemize}
    \item[(1)] $F^{r+\varepsilon,\mathcal{L} }_{p,q_1}(G)\hookrightarrow F^{r,\mathcal{L} }_{p,q_1}(G) \hookrightarrow F^{r,\mathcal{L} }_{p,q_2}(G) \hookrightarrow F^{r,\mathcal{L} }_{p,\infty}(G),$ $\varepsilon >0,$ $0\leq p\leq \infty,$ $0\leq q_1\leq q_2\leq \infty.$
    \item[(2)] $F^{r+\varepsilon,\mathcal{L} }_{p,q_1}(G) \hookrightarrow F^{r,\mathcal{L} }_{p,q_2}(G), $ $\varepsilon >0,$ $0\leq p\leq \infty,$ $1\leq q_2< q_1< \infty.$
    \item[(3)] $F^{r,\mathcal{L}}_{p,p}(G)=B^{r,\mathcal{L}}_{p,p}(G),$ $r\in \mathbb{R},$ $0<p\leq \infty.$
    \item[(4)] $B^{r,\mathcal{L}}_{p,\min\{p,q\}}(G)\hookrightarrow F^{r,\mathcal{L}}_{p,q}(G)\hookrightarrow B^{r,\mathcal{L}}_{p,\max\{p,q\}}(G) ,$ $0<p<\infty,$ $0<q\leq \infty.$
\end{itemize}
 \end{theorem} The proof is only an adaptation of the arguments presented in Triebel \cite{Triebel1983}. To clarify the relation between subelliptic Besov spaces and subelliptic Triebel-Lizorkin spaces we present their interpolation properties. We recall that if $X_0$ and $X_1$ are Banach spaces, intermediate spaces between $X_0$ and $X_1$ can be defined with the $K$-functional, which is defined by
\begin{equation}
K(f,t)=\inf\{\Vert f_{0}\Vert_{X_0}+t\Vert f_1 \Vert_{X_1}:f=f_0+f_1, f_0\in X_0,f_1\in X_1\},\,\,t\geq 0.
\end{equation}
If $0<\theta<1$ and $1\leq q<\infty,$ the real interpolation space $X_{\theta,q}:=(X_0,X_1)_{(\theta,q)}$ is defined by those vectors $f\in X_0+X_1$ satisfying 
\begin{equation}
\Vert f \Vert_{\theta,q}=\left(  \int_{0}^{\infty}  ( t^{-\theta}K(f,t))^{q}\frac{dt}{t}   \right)^{\frac{1}{q}}<\infty\,\, \text{if} \,\,q<\infty,
\end{equation}
and for $q=\infty$
\begin{equation}
\Vert f \Vert_{\theta,q} =\sup_{t>0}  t^{-\theta}K(f,t)<\infty.
\end{equation} 
Let us observe that if in Theorem 8.1 of \cite{NurRuzTikhBesov2017} we consider the case of any sub-Laplacian $\mathcal{L}$ instead of the Laplacian $\mathcal{L}_G,$  a similar result can be obtained with a proof by following   the lines of the proof presented there. In our case, we have the following interpolation theorem.
\begin{theorem}
Let $G$ be a compact Lie groups of dimension $n.$ If $0<r_0<r_1<\infty,$ $0<\beta_0,\beta_1,q\leq \infty$ and $r=r_{0}(1-\theta)+
r_{1}\theta,$ then
\begin{itemize}
    \item[(1)] $(B^{r_0,\mathcal{L}}_{p,\beta_0}(G),B^{r_1,\mathcal{L}}_{p,\beta_1}(G))_{(\theta,q)}=B^{r,\mathcal{L}}_{p,q}(G),$ $0<p<\infty.$
    \item[(2)] $(L^{r_0,\mathcal{L}}_{p}(G),L^{r_1,\mathcal{L}}_{p}(G))_{(\theta,q)}=B^{r,\mathcal{L}}_{p,q}(G),$ $1<p<\infty.$
    \item[(3)] $(F^{r_0,\mathcal{L}}_{p,\beta_0}(G),F^{r_1,\mathcal{L}}_{p,\beta_1}(G))_{(\theta,q)}=B^{r,\mathcal{L}}_{p,q}(G),$ $0<p<\infty.$
\end{itemize}
\end{theorem}

\section{Boundedness of pseudo-differential operators associated to H\"ormander classes}\label{boundedness}

In this section we study the action of pseudo-differential operators on subelliptic Sobolev and Besov spaces. Troughout this section, $X=\{X_1,\cdots,X_k\}$ is a system of vector fields satisfying the H\"ormander condition of order $\kappa$ (this means that their iterated commutators of length $\leq \kappa$ span the Lie algebra $\mathfrak{g}$ of $G$) and $\mathcal{L}\equiv \mathcal{L}_{\textnormal{sub}}=-\sum_{j=1}^k X_j^2$ is the associated positive sub-Laplacian.
\subsection{Boundedness properties on subelliptic Sobolev spaces}
In this subsection we consider the problem of the boundedness of pseudo-differential operators on subelliptic Sobolev spaces. Our analysis will be based on Theorem \ref{DelgadoRuzhanskyLppseudo} and Remark \ref{orderforpowers}. 
\begin{theorem}\label{SubellipticLp1}
Let us assume that $\sigma\in \mathscr{S}^{-\nu}_{\rho,\delta}(G)$ and $0\leq \delta<\rho\leq 1.$  Then $A\equiv \sigma(x,D)$ extends to a bounded operator from $L^{p,\mathcal{L}}_{\vartheta}(G)$ to $L^p(G)$ provided that
\begin{equation}
    n(1-\rho)|\frac{1}{p}-\frac{1}{2}|-\frac{\vartheta}{\kappa}\leq  \nu,
\end{equation} for $0< \rho<1,$ and $\nu=-\frac{\vartheta}{\kappa},$ for $\rho=1.$ In particular, if $\sigma\in \mathscr{S}^{0}_{\rho,\delta}(G),$ the operator $A\equiv \sigma(x,D)$ extends to a bounded operator from $L^{p,\mathcal{L}}_{\vartheta}(G)$ to $L^p(G)$ with
\begin{equation}
    { n\kappa(1-\rho)}|\frac{1}{p}-\frac{1}{2}|\leq \vartheta ,
\end{equation} and $\vartheta=0,$ for $\rho=1.$
\end{theorem}
\begin{proof}
Let us define the operator $T:=A\mathcal{M}_{-\vartheta}$ where $\mathcal{M}=(1+\mathcal{L})^\frac{1}{2}$ and $\mathcal{M}_{-\vartheta}=(1+\mathcal{L})^{-\frac{\vartheta}{2}}.$ From Remark   \ref{orderforpowers},  $\mathcal{M}_{-\vartheta}\in \Psi^{-\vartheta/\kappa}_{1/\kappa,0} $ and consequently $T\in \Psi^{-\vartheta/\kappa-\nu}_{1/\kappa,\delta} .$ Theorem \ref{DelgadoRuzhanskyLppseudo} implies that $T$ extends to a bounded operator on $L^p(G)$ provided that
\begin{equation}
    n(1-\rho)|\frac{1}{p}-\frac{1}{2}| \leq \nu+\frac{\vartheta}{\kappa}.
\end{equation} In this case, there exists $C>0$ independent of $g\in L^p(G),$ satisfying
\begin{equation}\label{SubellipticEstimate1}
    \Vert Tg\Vert_{L^p(G)}=\Vert A\mathcal{M}_{-\vartheta} g\Vert_{L^p(G)}\leq C\Vert g\Vert_{L^p(G)}.
\end{equation} In particular, if we replace $g$ in \eqref{SubellipticEstimate1} by $\mathcal{M}_{\vartheta} f,$ $f\in C^\infty(G),$ we obtain
\begin{equation}
    \Vert A f\Vert_{L^p(G)}\leq C\Vert \mathcal{M}_{\vartheta} f\Vert_{L^p(G)}=C\Vert  f\Vert_{L_{\vartheta}^{p,\mathcal{L}}(G)}.
\end{equation} Thus we end the proof.
\end{proof}
\begin{remark}\label{thecaseoflaplacian}
Let us observe that if in Theorem \ref{SubellipticLp1} we consider the Laplacian $\mathcal{L}_G$ instead any the sub-Laplacian $\mathcal{L}$ and we set $\kappa=1,$ a similar result can be obtained with a proof by following   the lines of the proof presented above for Theorem \ref{SubellipticLp1}. We present the corresponding result of the following way.
\end{remark} 
\begin{corollary}\label{SubellipticLp1cor}
Let us assume that $\sigma\in \mathscr{S}^{-\nu}_{\rho,\delta}(G)$ and $0\leq \delta<\rho\leq 1.$  Then $A\equiv \sigma(x,D)$ extends to a bounded operator from $L^{p}_{\vartheta}(G)$ to $L^p(G)$ provided that
\begin{equation}
    n(1-\rho)|\frac{1}{p}-\frac{1}{2}|-{\vartheta}\leq  \nu,
\end{equation} 
for $0<\rho<1,$ and $\nu=-\vartheta,$ for $\rho=1.$
In particular, if $\sigma\in \mathscr{S}^{0}_{\rho,\delta}(G),$ the operator $A\equiv \sigma(x,D)$ extends to a bounded operator from $L^{p}_{\vartheta}(G)$ to $L^p(G)$ with
\begin{equation}
    { n(1-\rho)}|\frac{1}{p}-\frac{1}{2}|\leq \vartheta,
\end{equation} for $0<\rho<1,$ and $\vartheta=0,$ for $\rho=1.$
\end{corollary}
We observe that Corollary \ref{SubellipticLp1cor} is only another formulation of Theorem \ref{DelgadoRuzhanskyLppseudo}.

\subsection{Boundedness properties on subelliptic Besov spaces}   
 Let us choose a $S^1_0$-partition of unity (or Littlewood-Paley partition of the unity) $\{\tilde{\psi}_{j}\}.$ So, we assume the following properties (see e.g.  Definition 2.1 of Garetto \cite{Garetto}),
\begin{itemize}
    \item  $\tilde\psi_0 \in \mathscr{D}(-\infty,\infty),$ $\tilde\psi_0(t)=1$ for $|t|\leq 1,$ and  $\tilde\psi_0(t)=0$ for  $|t|\geq 2.$
    \item For every $j\geq 1, $ $\tilde\psi_{j}(t)=\tilde\psi_0(2^{-j}t)-\tilde\psi_{0}(2^{-j+1}t):=\tilde\psi_1(2^{-j+1}t),$
    \item $|\tilde{\psi}^{(\alpha)}_j(t)|\leq C_{\alpha}2^{-\alpha j},\,\,\alpha\in\mathbb{N}_0,$ and
\item $ \sum_{j} \tilde{\psi}_{j}(t)=1, $ $t\in\mathbb{R}.$
\end{itemize}
If $\sigma\in \mathscr{S}^{-\nu}_{\rho,\delta}(G),$ $0\leq \delta<\rho\leq 1,$
in order to study the subelliptic Besov boundedness of the pseudo-differential operator $A\equiv \sigma(x,D)$, for every $\tilde\ell\in\mathbb{N},$ and $f\in C^\infty(G)$ we will decompose $\psi_{\tilde\ell}((1+\mathcal{L}_G)^{\frac{1}{2}})Af$ as
\begin{equation}
    \psi_{\tilde\ell}((1+\mathcal{L}_G)^{\frac{1}{2}})Af=\sum_{j,j'=0}^\infty \psi_{\tilde\ell}((1+\mathcal{L}_G)^{\frac{1}{2}})\tilde{\psi}_{j'}((1+\mathcal{L}_G)^{\frac{1}{2}})A\psi_j(D)f=\sum_{j,j'=0}^\infty T_{\tilde\ell,j,j'}f,
\end{equation} with $T_{\tilde\ell,j,j'}:=\psi_{\tilde\ell}((1+\mathcal{L}_G)^{\frac{1}{2}})\tilde{\psi}_{j'}((1+\mathcal{L}_G)^{\frac{1}{2}})A\psi_j(D) .$ By observing that the matrix symbol $\sigma_{T_{\tilde\ell,j,j'}}$ of every operator $T_{\tilde\ell,j,j'}$
satisfies the identity
\begin{align*}
    \sigma_{T_{\tilde\ell,j,j'}}(x,\xi)=\psi_{\tilde\ell}(\langle \xi\rangle)\tilde{\psi}_{j'}(\langle \xi\rangle) & \sigma_{  A\psi_j(D)}(x,\xi)\\=&\sigma_{  \psi_{\tilde\ell}((1+\mathcal{L}_G)^\frac{1}{2})\tilde{\psi}_{j'}((1+\mathcal{L}_G)^\frac{1}{2})A}(x,\xi)\psi_{j}( \xi),
\end{align*}
for all $x\in G$ and $[\xi]\in \widehat{G},$ where    $\sigma_{  A\psi_j(D)}(x,\xi)$ and $\sigma_{ \psi_{\tilde\ell}((1+\mathcal{L}_G)^\frac{1}{2})\tilde{\psi}_{j'}((1+\mathcal{L}_G)^\frac{1}{2})A }(x,\xi)$ are the corresponding symbols to the operators $ A\psi_j(D)$ and $  \psi_{\tilde\ell}((1+\mathcal{L}_G)^\frac{1}{2})\tilde{\psi}_{j'}((1+\mathcal{L}_G)^\frac{1}{2})A,$ we deduce that
\begin{equation}
    \textnormal{supp}(\sigma_{T_{\tilde\ell,j,j'}}(x,\xi))\subset \textnormal{supp}(\psi_{\tilde \ell}(\langle\xi\rangle)I_{d_\xi})\cap \textnormal{supp}(\psi_{j}(\xi))\cap \textnormal{supp} (\tilde{\psi}_{j'}(\langle\xi\rangle)I_{d_\xi}), 
\end{equation} for all $x\in G,$ and  $ [\xi]\in \widehat{G}.$ From the existence of  an integer $j_0\in\mathbb{N}$ satisfying $\tilde{\psi}_{j}\psi_{j'}=0$ for all $j,j'$ with $|j-j'|\geq j_0,$ we have the estimate
\begin{equation}
    \Vert \psi_{\tilde\ell}((1+\mathcal{L}_G)^\frac{1}{2})Af  \Vert_{L^p}\leq \sum_{j,j'=\tilde\ell-j_0}^{\tilde\ell+j_0}\Vert\psi_{\tilde\ell}((1+\mathcal{L}_G)^\frac{1}{2}) \tilde{\psi}_{j'}((1+\mathcal{L}_G)^\frac{1}{2}) A \psi_{j}(D)f \Vert_{L^p},
\end{equation}  for all  $\tilde\ell\in \mathbb{N}.$ From Lemma 4.11 of \cite{RuzhanskyDelgado2017}  applied to every operator  $  \psi_{\tilde\ell}((1+\mathcal{L}_G)^\frac{1}{2})\tilde{\psi}_{j'}((1+\mathcal{L}_G)^\frac{1}{2}),$ we have  $\Vert  \psi_{\tilde\ell}((1+\mathcal{L}_G)^\frac{1}{2}) \tilde{\psi}_{j'}((1+\mathcal{L}_G)^\frac{1}{2}) A \psi_{j}(D) \Vert_{L^p(G)}\leq C\Vert A \psi_{j}(D) f\Vert_{L^p(G)},$ where the positive constant $C>0$ is independent of $\tilde\ell,j,j'$ and $f.$ So, if additionally we use Theorem \ref{SubellipticLp1}, we can estimate the $B^{s}_{p,q}$-Besov norm of $Af$ as follows,
\begin{align*}
    \Vert Af\Vert_{ B^{s}_{p,q}(G) } &\asymp \left\Vert      \{ 2^{js}  \Vert  \psi_{\tilde\ell}((1+\mathcal{L}_G)^{\frac{1}{2}})A f\Vert_{L^{p}(G)} \}_{\tilde\ell\in\mathbb{N}}  \right\Vert_{\ell^q(\mathbb{N})}\\
&\leq  \left\Vert   \{   \sum_{j',j=\tilde\ell-j_0}^{\tilde\ell+j_0}2^{js}\Vert  \psi_{\tilde\ell}((1+\mathcal{L}_G)^{\frac{1}{2}})\tilde{\psi}_{j'}((1+\mathcal{L}_G)^{\frac{1}{2}})A\psi_{j}(D) f\Vert_{L^{p}(G)} \}_{\tilde\ell\in\mathbb{N}}\right\Vert_{\ell^q(\mathbb{N})}\\
&\asymp  \left\Vert   \{   \sum_{j',j=\tilde\ell-j_0}^{\tilde\ell+j_0}2^{js}\Vert  A\psi_{j}(D) f\Vert_{L^{p}(G)} \}_{\tilde\ell\in\mathbb{N}}\right\Vert_{\ell^q(\mathbb{N})}\\
&\lesssim  \left\Vert   \{   \sum_{j',j=\tilde\ell-j_0}^{\tilde\ell+j_0}2^{js}\Vert A \Vert_{\mathscr{B}(L^{p,\mathcal{L}}_\vartheta,L^p)}\Vert  \psi_{j}(D) f\Vert_{L_{\vartheta}^{p,\mathcal{L}}(G)} \}_{\tilde\ell\in\mathbb{N}}\right\Vert_{\ell^q(\mathbb{N})}\\
&\asymp \left\Vert   \{   \sum_{j',j=\tilde\ell-j_0}^{\tilde\ell+j_0}2^{j(s+\vartheta)}\Vert  \psi_{j}(D) f\Vert_{L^{p}(G)} \}_{\tilde\ell\in\mathbb{N}}\right\Vert_{\ell^q(\mathbb{N})}\\
&\lesssim (2j_0+1)^2\left\Vert   \{   2^{j(s+\vartheta)}\Vert  \psi_{j}(D) f\Vert_{L^{p}(G)} \}_{j\in\mathbb{N}}\right\Vert_{\ell^q(\mathbb{N})}
\asymp\Vert f\Vert_{ B^{s+\vartheta,\mathcal{L}}_{p,q}(G) },
\end{align*} provided that
\begin{equation}
    n(1-\rho)|\frac{1}{p}-\frac{1}{2}|-\frac{\vartheta}{\kappa}\leq  \nu<\frac{n(1-\rho)}{2}-\frac{\vartheta}{\kappa},
\end{equation}  for $0<\rho<1,$ and $\nu=-\frac{\vartheta}{\kappa},$ for $\rho=1.$   So, we summarise this analysis in the following result.
 \begin{theorem}\label{SubellipticBesov1}
Let us assume that $\sigma\in \mathscr{S}^{-\nu}_{\rho,\delta}(G)$ and $0\leq \delta<\rho\leq 1.$  Then $A\equiv \sigma(x,D)$ extends to a bounded operator from $B^{s+\vartheta,\mathcal{L}}_{p,q}(G)$   to $B^{s}_{p,q}(G)$ provided that
\begin{equation}
    n(1-\rho)|\frac{1}{p}-\frac{1}{2}|-\frac{\vartheta}{\kappa}\leq  \nu,
\end{equation} for $0<\rho<1,$ and $\nu=-\frac{\vartheta}{\kappa},$ for $\rho=1.$  In particular, if $\sigma\in \mathscr{S}^{0}_{\rho,\delta}(G),$ the operator $A\equiv \sigma(x,D)$ extends to a bounded operator from $B^{s+\vartheta,\mathcal{L}}_{p,q}(G)$   to $B^{s}_{p,q}(G)$  with
\begin{equation}
    { n\kappa(1-\rho)}|\frac{1}{p}-\frac{1}{2}|\leq \vartheta,
\end{equation} for $0<\rho<1,$ and $\vartheta=0,$ for $\rho=1.$
\end{theorem}
 \begin{remark}
 Same as in Remark \ref{thecaseoflaplacian}, if we consider the Laplacian on $G,$ $\mathcal{L}_G,$ instead of any sub-Laplacian, and we put $\kappa=1,$ a similar argument to the used in the proof of Theorem \ref{SubellipticBesov1} leads us to prove a similar result. We record it as follows.
 \end{remark}
 \begin{theorem}\label{SubellipticBesov2}
Let us assume that $\sigma\in \mathscr{S}^{-\nu}_{\rho,\delta}(G)$ and  $0\leq \delta<\rho\leq 1.$  Then $A\equiv \sigma(x,D)$ extends to a bounded operator from $B^{s+\vartheta}_{p,q}(G)$   to $B^{s}_{p,q}(G)$ provided that
\begin{equation}
    n(1-\rho)|\frac{1}{p}-\frac{1}{2}|-{\vartheta}\leq  \nu,
\end{equation} for $0<\rho<1,$ and $\nu=-\vartheta,$ for $\rho=1.$ In particular, if $\sigma\in \mathscr{S}^{0}_{\rho,\delta}(G),$ the operator $A\equiv \sigma(x,D)$ extends to a bounded operator from $B^{s+\vartheta}_{p,q}(G)$ to $B^{s}_{p,q}(G)$ with
\begin{equation}
    { n(1-\rho)}|\frac{1}{p}-\frac{1}{2}|\leq \vartheta ,
\end{equation} for $0<\rho<1,$ and $\vartheta=0,$ for $\rho=1.$
\end{theorem}
 Theorem \ref{SubellipticBesov2} shows that, for $0\leq \delta<\rho\leq 1,$ if we apply an operator $A$ to a function/distribution $f\in B^{s}_{p,q}(G), $ the function $Af$ has regularity order $s+\vartheta,$ $\vartheta={ n(1-\rho)}|\frac{1}{p}-\frac{1}{2}|.$ So, this result improves to ones recently  reported in \cite{Cardona1,Cardona2}   and \cite{Cardona3}.  Other properties of pseudo-differential operators related to the nuclearity notion on Besov spaces (associated to the Laplacian) are proved in  \cite{Cardona22}. The work \cite{CardonaRuzhansky2017I}
 by the authors provide a systematic investigation of the Besov spaces associated to Rockland operators on graded Lie groups.

\bibliographystyle{amsplain}

\end{document}